\documentclass[11pt, reqno]{amsart}

\usepackage{enumitem}

\usepackage{geometry}
\usepackage{color}
\usepackage{amssymb}
\usepackage{amsmath}
\usepackage{arydshln}
\usepackage{hyperref}
\usepackage{bbm}
\usepackage{mathrsfs}

\def\BibTeX{{\rm B\kern-.05em{\sc i\kern-.025em b}\kern-.08em
    T\kern-.1667em\lower.7ex\hbox{E}\kern-.125emX}}

\numberwithin{equation}{section}

\newcommand{\R}{\mathbb{R}}

\newcommand{\N}{\mathbb{N}}

\newcommand{\K}{\mathcal{K}}

\newcommand{\B}{\mathbb{B}}

\newcommand{\E}{\mathbb{E}}
\newcommand{\e}{\varepsilon}

\newcommand{\esssup}{\mathrm{ess}\sup}

\newcommand{\bfrac}[2]{\genfrac{}{}{0pt}{}{#1}{#2}}
\newcommand{\dm}{\mathrm{d}}
\newcommand{\vertiii}[1]{{\left\vert\kern-0.25ex\left\vert\kern-0.25ex\left\vert #1 
    \right\vert\kern-0.25ex\right\vert\kern-0.25ex\right\vert}}

\newtheorem{Theorem}{Theorem}[section]

\newtheorem{Corollary}[Theorem]{Corollary}
\newtheorem{Lemma}[Theorem]{Lemma}
\newtheorem{Remark}[Theorem]{Remark}
\newtheorem{Definition}[Theorem]{Definition}

\begin{document}

\title[Nonlinear perturbations of evolution systems in scales of Banach spaces]{Nonlinear perturbations of evolution systems in scales of Banach spaces}

\author{Martin Friesen}
\address[Martin Friesen]{Faculty of Mathematics and Natural Sciences, University of Wuppertal, Germany}
\email{friesen@math.uni-wuppertal.de}

\author{Oleksandr Kutoviy}
\address[Oleksandr Kutoviy]{Department of Mathematics, Bielefeld University, Germany}
\email{kutoviy@math.uni-bielefeld.de}

\date{\today}



\begin{abstract}
 A variant of the abstract Cauchy-Kovalevskaya theorem is considered. We prove existence and uniqueness of classical solutions to 
 the nonlinear, non-autonomous initial value problem
 \[
  \frac{\dm u(t)}{\dm t} = A(t)u(t) + B(u(t),t), \ \ u(0) = x
 \]
 in a scale of Banach spaces. Here $A(t)$ is the generator of an evolution system acting in a scale of Banach spaces and $B(u,t)$ obeys an Ovcyannikov-type bound.
 Continuous dependence of the solution with respect to $A(t)$, $B(u,t)$ and $x$ is proved.
 The results are applied to the Kimura-Maruyama equation for the mutation-selection balance model. This yields a new insight in the construction and
 uniqueness question for nonlinear Fokker-Planck equations related with interacting particle systems in the continuum.
\end{abstract}

\maketitle

\allowdisplaybreaks

\section{Introduction}

\subsection{Abstract Cauchy problem in a scale of Banach spaces}
Scales of Banach spaces are a commonly used tool for the study of different (linear and nonlinear) integro-differential equations.
They can be used to prove local existence, uniqueness and continuous dependence on the initial data when other methods do not apply. Applications to linear partial differential equations can be found in \cite{T08, T10, H13, BHP15},
whereas applications to semilinear partial differential equations are given in \cite{CAPS}.
Let us mention that such techniques have been successfully applied 
in \cite{U84} for the Boltzmann equation in space and time.
In the last years we also observe a growing interest to the study of Fokker-Planck equations related with infinite particle systems in the continuum, 
cf. \cite{FKO12, BKKK13, FK13, KK16}. 

The aforementioned applications are related to the time-inhomogeneous abstract Cauchy problem
\begin{align}\label{INTRO:00}
 \frac{\dm u(t)}{\dm t} = B(u(t),t),\ \  u(0) = x,
\end{align}
where $B$ is a (linear or nonlinear) operator in a scale of Banach spaces $\B = (\B_{\alpha})_{\alpha \in [\alpha_*, \alpha^*]}$,
$\Vert \cdot \Vert_{\alpha} \leq \Vert \cdot \Vert_{\alpha'}$ for $\alpha',\alpha \in [\alpha_*,\alpha^*]$ with $\alpha' < \alpha$. 
It is assumed that
$B: \B_{\alpha'} \times [0,T] \longrightarrow \B_{\alpha}$ is continuous and satisfies, for all $\alpha' < \alpha$ and $u,v \in \B_{\alpha'}$ with 
$\Vert u\Vert_{\alpha'}, \Vert v \Vert_{\alpha'} \leq r$,
\[
 \Vert B(u,t) - B(v,t) \Vert_{\alpha} \leq \frac{C}{\alpha - \alpha'} \Vert u - v \Vert_{\alpha'}.
\]
The general approach to equations of the form \eqref{INTRO:00} goes back to Ovcyannikov \cite{OVC74,OVC80, OVC13}, Nirenberg \cite{NIR72} and Nishida \cite{NIS77}. 
Later on, Zabre{\u\i}ko \cite{ZA89} discussed the more general evolution equation
\[
 \frac{\dm^m u}{\dm t^m}(t) = B(u(t),t), \ u(0) = u_0, \ \cdots, \ u^{(m-1)}(0) = u_{m-1},
\]
where $m \in \N$. A construction of solutions to \eqref{INTRO:00} in a suitable space of continuous $\B$-valued functions was considered by Safonov \cite{SAF95}. 
The study of fractional time derivatives with linear operator $B$ was recently initiated in \cite{KK16a}. 

In this work we investigate the linear perturbation
\begin{align}\label{NONLINEAR:08}
 \frac{\dm u(t)}{\dm t} = A(t)u(t) + B(u(t),t), \ \ u(0) = x \in \B_{\alpha_*}.
\end{align}
A linear version of \eqref{NONLINEAR:08} was considered in \cite{F15} and more recently also in \cite{F19}.
We show that, under some suitable conditions, \eqref{NONLINEAR:08} is equivalent to the integral equation
\begin{align}\label{NONLINEAR:00}
 u(t) = U(t,0)x + \int \limits_{0}^{t}U(t,s)B(u(s),s)\dm s.
\end{align}
Here, $U(t,s)$ is an evolution system on $\B$ having $A(t)$ as its generator in some extended sense.
Using a modified version of the Banach space introduced in \cite{SAF95}, we prove by a fix point argument the existence and uniqueness of classical solutions to \eqref{NONLINEAR:00} and consequently to \eqref{NONLINEAR:08}.
Afterwards, we study stability of the solution $u$ with respect to $x,A(t)$ and $B(u,t)$. 

\subsection{Kimura-Maruyama equation}
We apply our results to the Kimura-Maruyama equation for the (epistatic) mutation-selection balance model. 
Such model was introduced in \cite{KM66} to describe aging of a population. A mathematical analysis and additional information on the 
biological background are given in \cite{SEW05}. The dynamics is described by a family
of probability measures $(\mu_t)_{t \in [0,T)}$ on the space of locally finite configurations $\Gamma$ with underlying complete, metric space $X$. Such a family is assumed to solve the nonlinear Fokker-Planck equation
\begin{align}\label{INTRO:01}
 \frac{\dm}{\dm t}\langle F, \mu_t \rangle = \langle L(t,\mu_t)F, \mu_t \rangle, \ \ \mu_t|_{t=0} = \mu_0,
\end{align}
where $\langle F, \mu \rangle := \int_{\Gamma}F(\gamma)\dm \mu(\gamma)$ and the (nonlinear) Markov operator $L(t,\mu)$ is given by
\begin{align}\label{EQ:105}
 L(t,\mu)F(\gamma) = \int \limits_{X}(F(\gamma \cup \{x\}) - F(\gamma))a(t,x)\sigma(\dm x) - \Phi(t,\gamma)F(\gamma) + \langle \Phi(t,\cdot),\mu\rangle F(\gamma).
\end{align}
Appearance of new mutations is described by a free birth process on $\Gamma$ (the so-called Sourgailis process) represented by the first term in \eqref{EQ:105},
where $\sigma$ is a $\sigma$-finite Borel measure on $X$ and $a(t,x) \geq 0$ is supposed to be a continuous bounded function on $\R_+ \times X$. The killing term $-\Phi(t,\gamma)F(\gamma)$ in \eqref{EQ:105} describes the selection mechanism for mutations where
\[
 \Phi(t,\gamma) = \sum \limits_{x \in \gamma}h(t,x) + \frac{1}{2}\sum \limits_{x \in \gamma}\sum \limits_{y \in \gamma \backslash \{x\}}\psi(t,x,y), \ \ t \geq 0, \ \gamma \in \Gamma
\]
is the selection cost functional and $h,\psi$ are assumed to be integrable, continuous, bounded,
and $\psi(t,x,y) = \psi(t,y,x)$ for all $t \geq 0$, $x,y \in X$.
The last term in \eqref{EQ:105} guarantees that the evolution of states obtained from the Fokker-Planck equation \eqref{INTRO:01} preserves probability. 

For the general theory of nonlinear Markov evolutions we refer to \cite{KOL10}. A detailed analysis of this model can be found in \cite{SEW05, KKO08, KKMP13}.
In the time-homogeneous case, with $a = 1$, a solution was constructed by the Feynmann-Kac formula and its behaviour when
$t \to \infty$ was studied in \cite{SEW05}. 
Under additional assumptions on $\psi$ and $h$ it was shown in \cite{KKO08} that the invariant measure is a Gibbs measure with energy $\Phi(\gamma)$.
However, uniqueness to \eqref{INTRO:01} was only shown for the non-epistatic case $\psi = 0$. In this work we provide existence and uniqueness of classical solutions 
to the associated hierarchical equations of correlation functions which can be seen as an Markov analogue of the BBGKY-hierarchy from physics. As a consequence, we are able to prove the uniqueness for a certain class of solutions to \eqref{INTRO:01} in the epistatic case $\psi \neq 0$.

\subsection{Structure of the work}
This work is organized as follows. Section 2 is devoted to the construction and uniqueness of classical solutions to \eqref{NONLINEAR:08}.
Stability of the solutions is studied in Section 3. The obtained results are applied in Section 4 to the Kimura-Maruyama equation \eqref{INTRO:01}.

\section{Existence and uniqueness}

\subsection{Scales of Banach spaces}
Let $\B = (\B_{\alpha},\Vert \cdot \Vert_{\alpha})_{\alpha_* \leq \alpha \leq \alpha^*}$ be a scale of Banach spaces on the interval $[\alpha_*,\alpha^*] \subset \R$.
That is for any $\alpha', \alpha \in [\alpha_*,\alpha^*]$ with $\alpha' < \alpha$: $\Vert \cdot \Vert_{\alpha} \leq \Vert \cdot \Vert_{\alpha'}$
and $\B_{\alpha'} \subset \B_{\alpha}$ hold. Let $i_{\alpha' \alpha} \in L(\B_{\alpha'}, \B_{\alpha})$
be the corresponding embedding operator,
where $L(\B_{\alpha'}, \B_{\alpha})$ denotes the space of all bounded linear operators from $\B_{\alpha'}$ to $\B_{\alpha}$. We set
$B_r^{\alpha}(x) = \{ y \in \B_{\alpha}\ | \ \Vert y-x\Vert_{\alpha} \leq r\}$ with
$r \in (0,\infty)$ and $x \in \B_{\alpha}$. Similarly, let $B_{\infty}^{\alpha}(x) := \B_{\alpha}$, whenever $\alpha \in [\alpha_*,\alpha^*]$ and $x \in \B_{\alpha}$. 
\begin{Definition}
The nonlinear operator $L$ in a scale of Banach spaces $\B$ is, by definition, a collection of mappings 
$L = (L_{\alpha'\alpha})_{\alpha_* \leq \alpha' < \alpha \leq \alpha^*}$ such that there exists $r \in (0,\infty]$  and $x \in \B_{\alpha_*}$
for which the following two properties hold
\begin{enumerate}
 \item[(i)] For all $\alpha', \alpha \in [\alpha_*,\alpha^*]$, $\alpha' < \alpha$ the map $L_{\alpha' \alpha}$ maps $B_r^{\alpha'}(i_{\alpha_* \alpha'}(x))$
 to $\B_{\alpha}$.
 \item[(ii)] For all $\alpha, \alpha',\alpha'' \in [\alpha_*,\alpha^*]$ with $\alpha' < \alpha < \alpha''$ we have
 \begin{align}\label{NONLINEAR:05}
  L_{\alpha' \alpha''}y = i_{\alpha \alpha''}L_{\alpha' \alpha}y = L_{\alpha \alpha''}i_{\alpha' \alpha}y, \ \ \forall y \in B_r^{\alpha'}(i_{\alpha_* \alpha'}(x)).
 \end{align}
\end{enumerate}
\end{Definition}
We call the operator $L$ a bounded linear operator in the scale $\B$ if $L_{\alpha'\alpha}$ is defined on the whole $\B_{\alpha'}$ (that is $r = \infty$), $L_{\alpha' \alpha} \in L(\B_{\alpha'}, \B_{\alpha})$ and \eqref{NONLINEAR:05} holds for all $y \in \B_{\alpha'}$. 
For two linear operators in the scale $\B$, say $L$ and $K$, the composition $LK$ is defined by
\begin{align}\label{NONLINEAR:06}
 (LK)_{\alpha' \alpha} = L_{\beta \alpha}K_{\alpha' \beta},
\end{align}
where $\beta \in (\alpha', \alpha)$. Above definition yields again a bounded linear operator $LK$ in the scale $\B$.
It is worth noting that definition \eqref{NONLINEAR:06} does not depend on $\beta$, because of \eqref{NONLINEAR:05}. In the same way we may also define $LK$, if $L$ is a bounded linear operator and $K$ a nonlinear operator in the scale $\B$.

In the following we always work with linear and nonlinear operators in the scale $\B$ and omit the subscripts $\alpha'\alpha$ when no confusion can arise.
For simplicity of notation, we let $x=y$, $x \in \B_{\alpha}$, $y \in \B_{\alpha'}$ stand for $i_{\alpha' \alpha}y = x$ and,
since no confusion may arise, we let $B_r^{\alpha}(y)$ stand for $B_r^{\alpha}(i_{\alpha'\alpha}(y))$.

\subsection{Framework}
Fix $x \in \B_{\alpha_*}$ and $T > 0$. 
\begin{enumerate}
 \item[A1.] There exists a family of bounded linear operators $(U(t,s))_{0 \leq s \leq t < T}$ in the scale $\B$
 such that for $0 \leq s \leq r \leq t$
 \[
  U(t,t) = 1,\ \  U(t,r)U(r,s) = U(t,s)
 \]
 holds in the sense of \eqref{NONLINEAR:06} and,
 $(t,s) \longmapsto U(t,s) \in L(\B_{\alpha'},\B_{\alpha})$ is strongly continuous for all $\alpha',\alpha$ with $\alpha_* \leq \alpha' < \alpha \leq \alpha^*$.
\end{enumerate}
For the nonlinear part $B(u,t)$ we suppose that there exists $r \in (0,\infty]$ such that $(B(\cdot,t))_{t \in [0, T)}$ is a family of nonlinear operators in the scale $\B$ satisfying
\begin{enumerate}
 \item[B1.] For all $\alpha', \alpha$ with $\alpha_* \leq \alpha' < \alpha \leq \alpha^*$
 \[
  B_r^{\alpha'}(x) \times  [0, T ) \ni (u,t) \longmapsto B(u,t) \in \B_{\alpha}
 \]
 is jointly continuous.
\end{enumerate}
Next we give a rigorous definition of a solution to \eqref{NONLINEAR:00}. As we want to consider a solution in different spaces $\B_{\alpha}$, it is reasonable to expect that the corresponding intervals on which it is defined do also depend on $\alpha$. Moreover, as $\B_{\alpha'} \subset \B_{\alpha}$, for $\alpha' < \alpha$, it is also reasonable to expect that these intervals increase in $\alpha$. Both effects are catched by 
\[
  T(\alpha, \alpha_0; T') = \frac{\alpha - \alpha_0}{\alpha^* - \alpha_0}T', \qquad \alpha_* \leq \alpha_0 \leq \alpha \leq \alpha^*.
\]
Note that $T(\alpha, \alpha_0; T')$ is increasing in $\alpha$ such that $T(\alpha_0, \alpha_0;T') = 0$ and $T(\alpha^*, \alpha_0;T') = T'$. 
\begin{Definition}\label{DEF:00}
 Take $\alpha_0 \in [\alpha_*, \alpha^*)$  and $T' \in (0,T]$. A solution to \eqref{NONLINEAR:00} on the interval $[0,T')$
 in the scale $(\B_{\alpha})_{\alpha \in [\alpha_0, \alpha^*]}$ is a function $u: [0, T(\alpha^*, \alpha_0;T')) \longrightarrow \B_{\alpha^*}$
 such that, for all $\alpha \in (\alpha_0, \alpha^*]$, the following properties are satisfied:
 \begin{enumerate}
  \item[(i)] $u|_{[0, T(\alpha, \alpha_0;T'))} \in C\left( \left[ 0, T(\alpha, \alpha_0;T')\right); \B_{\alpha}\right)$ and $u(t) \in B_r^{\alpha}(x)$ for each $t \in  \left[ 0, T(\alpha, \alpha_0;T'))\right)$.
  \item[(ii)] For each $t \in  \left[ 0, T(\alpha, \alpha_0;T')\right)$
  we find $\alpha', \alpha'' \in [\alpha_0, \alpha)$ satisfying $\alpha' < \alpha''$
  and $t \in \left[ 0, T(\alpha', \alpha_0;T') \right)$ such that the following identity holds in $\B_{\alpha}$
  \begin{align}\label{NONLINEAR:07}
   u(t) = U_{\alpha_* \alpha}(t,0)x + \int \limits_{0}^{t} U_{\alpha'' \alpha}(t,s)B_{\alpha' \alpha''}(u(s),s) ds.
  \end{align}
 \end{enumerate}
\end{Definition}
It is worthwile to mention that, by \eqref{NONLINEAR:05} and \eqref{NONLINEAR:06}, this notion of a solution to \eqref{NONLINEAR:00} does not depend on the particular choice of $\alpha', \alpha''$,
at least as long as $0 \leq t < T(\alpha', \alpha_0;T')$ is satisfied. Hence in order to simplify notation, 
we omit the subscripts $\alpha', \alpha''$ in \eqref{NONLINEAR:07} whenever no confusion may arise. While $\alpha$ plays the role of the terminal Banach space of the solution,
the parameter $\alpha_0$ describes the minimal Banach space on which the solution can be defined.

\subsection{The main existence and uniqueness statement} 
Fix $x \in \B_{\alpha_*}$, $T > 0$ and assume that 
$U(t,s)$ and $B(u,t)$ are given as in A1 and B1. 
The following summarizes our main assumptions:
\begin{enumerate}
 \item[A2.] There exist constants $C_1 > 0$ and $\beta \in [0,\frac{1}{2})$ such that for all $\alpha', \alpha$ with $\alpha_* \leq \alpha' < \alpha \leq \alpha^*$
 \[
  \Vert U(t,s)\Vert_{\alpha' \alpha} \leq \frac{C_1}{(\alpha- \alpha')^{\beta}}, \ \ 0 \leq s \leq t < T
 \]
 holds.
 \item[A3.] For each $\alpha \in (\alpha_*, \alpha^*]$ there exists $C(x,\alpha) > 0$ such that
 \[
  \Vert U(t,0)x - U(s,0)x \Vert_{\alpha} \leq C(x,\alpha) |t-s|, \qquad 0\leq s, t \leq \frac{\alpha - \alpha_*}{\alpha^* - \alpha_*}T
 \]
 is satisfied.
\end{enumerate}
\begin{Remark}
 In many cases one has the stronger condition $U(t,s) \in L(\B_{\alpha})$ for all $\alpha \in [\alpha_*,\alpha^*]$, 
 e.g. if $A(t)$ generates an exponentially bounded evolution system (see e.g. \cite{PAZ83} for additional details on such evolution systems). 
 However, applications we have in mind rather deal with scales of 
 weighted $L^{\infty}$-type spaces. In such a case it is unusual to have $U(t,s) \in L(\B_{\alpha})$, but conditions A1 -- A3 may be still satisfied.
\end{Remark}
For the nonlinear part we suppose that:
\begin{enumerate}
 \item[B2.] There exists a constant $C_2 > 0$ such that for all $\alpha',\alpha$ with $\alpha_* \leq \alpha' < \alpha \leq \alpha^*$, 
 $t \in [0,T)$ and any $u,v \in B_r^{\alpha'}(x)$
 \[
  \Vert B(u,t) - B(v,t) \Vert_{\alpha} \leq \frac{C_2}{(\alpha - \alpha')^{1 - \beta}}\Vert u - v\Vert_{\alpha'}
 \]
 holds.
 \item[B3.] There exists a constant $C_3 > 0$ such that for all $t \in [0,T)$ and all $\alpha \in (\alpha_*, \alpha^*]$
 \[
  \Vert B(x,t) \Vert_{\alpha} \leq \frac{C_3}{\alpha - \alpha_*}
 \]
 is satisfied.
\end{enumerate}
\begin{Remark}
 Suppose that we have given a family of operators $\widetilde{B} = (\widetilde{B}_{\alpha})_{\alpha \in [\alpha_*,\alpha^*]}$ with the properties
 \begin{enumerate}
  \item[(i)] $\widetilde{B}_{\alpha}(u,t) \in \B_{\alpha}$ for all $u \in \B_{\alpha}$, all $t \in [0,T)$ and all $\alpha \in [\alpha_*, \alpha^*]$.
  \item[(ii)] For all $\alpha',\alpha \in [\alpha_*,\alpha^*]$ with $\alpha' < \alpha$ we have
  \[
   \widetilde{B}_{\alpha}(i_{\alpha'\alpha}(u),t) = i_{\alpha'\alpha}\widetilde{B}_{\alpha'}(u,t), \ \ u \in \B_{\alpha'}, \ \ t \in [0,T).
  \]
  \item[(iii)] There exists a constant $C_4 > 0$ such that for all $t \in [0,T)$
  \[
   \| \widetilde{B}_{\alpha}(u,t) - \widetilde{B}_{\alpha}(v,t) \|_{\alpha} \leq C_4 \| u-v\|_{\alpha}, \ \ u,v \in B_r^{\alpha}(x).
  \]
 Moreover $\widetilde{B}_{\alpha}$ is jointly continuous on $B_r^{\alpha}(x) \times  [0,T)$ for any $\alpha \in [\alpha_*,\alpha^*]$.
 \end{enumerate}
 Then $\mathcal{B}(u,t) := B(u,t) + \widetilde{B}(u,t)$ is a family of nonlinear operators in the scale $\B$ which satisfies assumptions B1 -- B3.
\end{Remark}
Next we define the largest time interval on which we are able to construct a solution to \eqref{NONLINEAR:00}.
 Take $\gamma \in (\beta, 1 - \beta)$, $\alpha_0 \in (\alpha_*, \alpha^*)$, $r' \in (0, r]$ and define
\begin{align*}
 \lambda_0(x,\alpha_0, \gamma, r', T) = \lambda_0 := \max \bigg \{ &\frac{\alpha^* - \alpha_*}{T}, \frac{2^{2\gamma+1 - \beta}C_1C_2}{\gamma-\beta}, \frac{4^{1-\beta}C_2(\alpha^*- \alpha_0)^{\beta}}{\gamma(1+\Vert x \Vert_{\alpha_*})} + \frac{2^{2+\gamma}C_1C_2}{\gamma},
 \\ & \frac{C(x)(\alpha^*- \alpha_0)}{r'} + \frac{C_1(\frac{C_3}{\alpha_0 - \alpha_*} + C(x))(\alpha^* - \alpha_0)(1+\Vert x \Vert_{\alpha_*})}{(1-\gamma)r'} \bigg\},
\end{align*}
where $C(x) = C(x,\alpha_0)$ is given by condition A3, 
and we use the convention $1/\infty := 0$.
Although this constant seems to be quite complicated, 
we provide a precise formula since it is used in the next section to study stability of the solution.
Then $[0, \frac{\alpha^* - \alpha_0}{\lambda_0})$ describes the largest possible time interval on which we may construct a solution. In this work we use
\[
 \lambda > \lambda_0\ \text{ and }\ T' = \frac{\alpha^* - \alpha_0}{\lambda}\ \text{ so that }\
 T(\alpha, \alpha_0, T') = \frac{\alpha - \alpha_0}{\lambda}.
\]
To simplify the proofs, we will write $\frac{\alpha - \alpha_0}{\lambda}$ instead of $T(\alpha, \alpha_0;T')$. The following is our main result, which provides existence and uniqueness of solutions to \eqref{NONLINEAR:00}.
\begin{Theorem}\label{NONLINEARTH:00}
 Suppose that A1 -- A3 and B1 -- B3 are satisfied for $T > 0$, $x \in \B_{\alpha_*}$ and either $(\beta = 0, r \in (0,\infty])$ or
 $(\beta \in (0,\frac{1}{2}), r = \infty)$. 
 Fix $\alpha_0 \in (\alpha_*, \alpha^*)$ and $\gamma \in (\beta, 1 - \beta)$. Then for each $\lambda > \lambda_0(x,\alpha_0,\gamma,r,T)$ there exists a solution $u_{\alpha_0, \lambda, \gamma}$ to \eqref{NONLINEAR:00} on the interval $[0, \frac{\alpha^* - \alpha_0}{\lambda})$ in the scale $(\B_{\alpha},\Vert \cdot\Vert_{\alpha})_{\alpha \in [\alpha_0, \alpha^*]}$ satisfying
   \begin{align}\label{EQ:03}
   \sup \limits_{\alpha \in (\alpha_0, \alpha^*]} \sup \limits_{0 \leq t < \frac{\alpha - \alpha_0}{\lambda}} (\alpha - \alpha_0 - \lambda t)^{\gamma}\| u_{\alpha_0, \lambda, \gamma}(t) \|_{\alpha} < \infty.
  \end{align}
  Moreover, this solution also satisfies the a-priori estimate
  \begin{align}\label{EQ:106}
   (\alpha - \alpha_0 - \lambda t)^{\gamma}\Vert B(u_{\alpha_0,\lambda,\gamma}(t),\tau)\Vert_{\alpha} \leq  \left(\frac{C_3}{\alpha_0 - \alpha_*} + C(x,\alpha_0)\right)(\alpha^* - \alpha_0)^{\gamma}(1+\Vert x \Vert_{\alpha_*}),
  \end{align} 
  where $0 \leq \tau \leq \frac{\alpha^* - \alpha_0}{\lambda}$, 
  $0 \leq t < \frac{\alpha - \alpha_0}{\lambda}$ and $\alpha \in (\alpha_0, \alpha^*]$. 
\end{Theorem}
Note that for each $t \in [0, \frac{\alpha - \alpha_0}{\lambda})$ we find $\alpha' \in (\alpha_0, \alpha)$ 
such that $t < \frac{\alpha' - \alpha_0}{\lambda} < \frac{\alpha - \alpha_0}{\lambda}$ holds. This shows that $u_{\alpha_0, \lambda, \gamma}(t) \in \B_{\alpha'}$ and hence $\Vert B(u_{\alpha_0, \gamma}(t),\tau)\Vert_{\alpha}$ in \eqref{EQ:106} makes sense.
Assertion (b) provides uniqueness for solutions to \eqref{NONLINEAR:00}. From this we can deduce the following Corollary.
\begin{Corollary}
 Suppose that A1 -- A3 and B1 -- B3 are satisfied for $T > 0$, $x \in \B_{\alpha_*}$ and either $(\beta = 0, r \in (0,\infty])$ or
 $(\beta \in (0,\frac{1}{2}), r = \infty)$. 
 Fix $\alpha_0 \in (\alpha_*, \alpha^*)$ and $\gamma \in (\beta, 1 - \beta)$. Take $r' \in (0, r]$ such that $r' < \infty$ and let $\lambda > \lambda_0(x ,\alpha_0, \gamma, r',T)$. 
  Then there exists exactly one solution $u$ to \eqref{NONLINEAR:00} on the interval $[0, \frac{\alpha^* - \alpha_0}{\lambda})$ in the scale $(\B_{\alpha}, \| \cdot\|)_{\alpha \in [\alpha_0, \alpha]}$ which satisfies
  \[
   \sup \limits_{\alpha \in (\alpha_0, \alpha^*]} \sup \limits_{0 \leq t < \frac{\alpha - \alpha_0}{\lambda}} \| u(t) - x \|_{\alpha} \leq r'.
  \]
\end{Corollary}
\begin{proof}
 Observe that $u$ satisfies
 \begin{align*}
 (\alpha - \alpha_0 - \lambda t)^{\gamma}\Vert u(t)\Vert_{\alpha} 
 &\leq (\alpha - \alpha_0 - \lambda t)^{\gamma}\Vert u(t) - x\Vert_{\alpha} + (\alpha - \alpha_0 - \lambda t)^{\gamma} \Vert x \Vert_{\alpha}
 \\ &\leq (\alpha^* - \alpha_0)^{\gamma}r' + (\alpha^* - \alpha_0)^{\gamma}\Vert x \Vert_{\alpha_*} < \infty,
 \end{align*}
 which shows that $u$ satisfies \eqref{EQ:03}.
\end{proof}

\subsection{Proof of Theorem \ref{NONLINEARTH:00}}
As a first step we reformulate \eqref{NONLINEAR:00} into a fixed point problem 
in certain weighted Banach spaces. 
Since $\alpha_0 \in (\alpha_*, \alpha^*)$ is fixed,
we omit the corresponding subscript in the definitions and arguments given below.
Denote by $S^{\gamma} = S^{\gamma}(\lambda)$ the Banach space of all functions 
$u: [0, T(\alpha^*)) \longrightarrow \B_{\alpha^*}$ which satisfy
\[
 \Vert u \Vert^{(\gamma)} = \sup \limits_{\bfrac{0 \leq t < T(\alpha)}{\alpha \in (\alpha_0, \alpha^*]}}\ (\alpha - \alpha_0 - \lambda t)^{\gamma} \Vert u(t)\Vert_{\alpha} < \infty
\]
with $T(\alpha) := \frac{\alpha - \alpha_0}{\lambda}$, and
\begin{align}\label{EQ:107}
 u|_{[0,T(\alpha))} \in C([0,T(\alpha));\B_{\alpha}), \qquad \forall \alpha \in (\alpha_0, \alpha^*].
\end{align} 
Let $u \in S^{\gamma}$, then $\| u - x \|^{(0)} \leq r$
is equivalent to
\[
u(t) \in B_r^{\alpha}(x), \qquad  \alpha \in (\alpha_0, \alpha^*],\  \ t \in [0,T(\alpha)).
\]
For such $u$ we define
\begin{align}\label{NONLINEAR:01}
 \mathcal{T}(u)(t) := \int \limits_{0}^{t}U(t,s)B(u(s),s)\dm s.
\end{align}
Then $\mathcal{T}u|_{[0,T(\alpha))} \in C( [0,T(\alpha)); \B_{\alpha})$ for all $\alpha \in (\alpha_0, \alpha^*]$. Indeed, take $\alpha \in (\alpha_0, \alpha]$ and $t \in [0,T(\alpha))$. Then we find $\alpha', \alpha'' \in [\alpha_0, \alpha)$ with
$\alpha' < \alpha''$ and $t \in [0,T(\alpha'))$. Hence 
\[
 U_{\alpha'' \alpha}(t,s)B_{\alpha' \alpha''}(u(s),s) \in \B_{\alpha}, \qquad s \in [0,t],
\]
and we conclude that the integral \eqref{NONLINEAR:01} exists in $\B_{\alpha}$.
Similarly to property (ii) from Definition \ref{DEF:00}, we find that
$\mathcal{T}(u)$ is independent of the particular choice of $\alpha', \alpha''$. The next Lemma is immediate, a proof is therefore omitted.
\begin{Lemma} \label{LEMMA:03}
 Let $u \in S^{\gamma}$. The following are equivalent:
 \begin{enumerate}
  \item[(i)] $u$ is a solution to \eqref{NONLINEAR:00} in the scale $(\B_{\alpha})_{\alpha\in [\alpha_0, \alpha^*]}$
  \item[(ii)] $u$ satisfies $u(t) \in B_r^{\alpha}(x)$
for each $t \in [0,T(\alpha))$ and each $\alpha \in (\alpha_0, \alpha^*]$,
and $u = U(\cdot,0)x + \mathcal{T}(u)$.
  \end{enumerate}
\end{Lemma}
In view of Lemma \ref{LEMMA:03} we proceed to show that 
$u = U(\cdot,0)x + \mathcal{T}(u)$ has for all $\lambda > \lambda_0(x,\alpha_0, \gamma,r,T)$ a unique solution in $S^{\gamma}$ satisfying $u(t) \in B_r^{\alpha}(x)$
for each $t \in [0,T(\alpha))$ and each $\alpha \in (\alpha_0, \alpha^*]$.
For this purpose we first show that $\mathcal{T}$ is a contraction operator.
\begin{Lemma}\label{LEMMA:04}
 For all $u,v \in S^{\gamma}$ with $\| u - x \|^{(0)}, \| v - x\|^{(0)} \leq r$ we have
 \[
  \Vert \mathcal{T}(u) - \mathcal{T}(v) \Vert^{(\gamma)} \leq \frac{\lambda_0}{\lambda} \Vert u - v \Vert^{(\gamma)},
 \]
\end{Lemma}
\begin{proof}
Let $u \in S^{\gamma + 1-\beta}$, fix $\alpha \in (\alpha_0,\alpha^*]$ and define $\alpha'$ by $\alpha = \alpha' + \frac{\rho(t)}{2}$ 
with $\rho(t) = \alpha - \alpha_0 - \lambda t$ and $0 \leq t < T(\alpha)$. Then
 \[
  \alpha - \alpha' = \frac{\rho(t)}{2} = \alpha' - \alpha_0 - \lambda t
 \]
 and using A2 we obtain for each $u \in S^{\gamma + 1 - \beta}$,
\begin{align*}
 (\alpha - \alpha_0 - \lambda t)^{\gamma}\left \Vert \int \limits_{0}^{t}U(t,s)u(s)\dm s\right \Vert_{\alpha} &\leq \frac{C_1}{(\alpha - \alpha')^{\beta}}(\alpha - \alpha_0 - \lambda t)^{\gamma}\int \limits_{0}^{t}\Vert u(s)\Vert_{\alpha'}\dm s
 \\ &\leq 2^{\beta}C_1\rho(t)^{\gamma - \beta}\int \limits_{0}^{t}(\alpha' - \alpha_0 - \lambda s)^{-(\gamma + 1-\beta)}\dm s\Vert u \Vert^{(\gamma + 1-\beta)}
 \\ &\leq \frac{2^{\beta}C_1}{(\gamma - \beta)\lambda}\rho(t)^{\gamma - \beta}\Vert u \Vert^{(\gamma + 1-\beta)}(\alpha' - \alpha_0 - \lambda t)^{-\gamma + \beta}
 \\ &= \frac{2^{\gamma}C_1}{(\gamma - \beta)\lambda}\Vert u \Vert^{(\gamma + 1-\beta)}.
\end{align*}
 Taking the supremum over all $t \in [0, T(\alpha))$ and $\alpha \in (\alpha_0, \alpha^*]$ yields
\begin{align}\label{NONLINEAR:10}
 \left \Vert \int \limits_{0}^{\bullet}U(\cdot,s)u(s)\dm s\right\Vert^{(\gamma)} \leq \frac{2^{\gamma}C_1}{(\gamma - \beta) \lambda}\Vert u \Vert^{(\gamma + 1-\beta)}.
\end{align}
Analogously to \eqref{NONLINEAR:10} one shows that, for $u,v \in S^{\gamma}$ with $\| u - x \|^{(0)}, \| v - x\|^{(0)} \leq r$,
\begin{align}\label{NONLINEAR:04}
 \Vert B(u(\cdot),\cdot) - B(v(\cdot),\cdot)\Vert^{(\gamma + 1 - \beta)} \leq 2^{\gamma+ 1 - \beta}C_2\Vert u - v\Vert^{(\gamma)},
\end{align}
see also \cite[Lemma 3.1]{SAF95} for a similar estimate.
 For the same $u,v$ we obtain from \eqref{NONLINEAR:10} and \eqref{NONLINEAR:04} 
\begin{align*}
 \Vert \mathcal{T}(u) - \mathcal{T}(v)\Vert^{(\gamma)} &\leq \frac{2^{\gamma}C_1}{(\gamma-\beta) \lambda}\Vert B(u(\cdot),\cdot) - B(v(\cdot), \cdot)\Vert^{(\gamma+ 1 - \beta)} 
 \\ &\leq \frac{2^{2\gamma+ 1 - \beta}C_1C_2}{(\gamma-\beta) \lambda}\Vert u - v \Vert^{(\gamma)} \leq \frac{\lambda_0}{\lambda}\Vert u -v\Vert^{(\gamma)}
\end{align*}
 which proves the assertion.
\end{proof}
Existence of a solution is obtained by the usual fixed point argument.
However, in order to guarantee that such iterations are well-defined we have to restrict $\mathcal{T}$ to a proper subspace of $S^{\gamma}$
defined as follows: Define
\begin{align}\label{NONLINEAR:02}
 S_x : = \left\{ u \in S^{\gamma}\ \bigg| \ \Vert u - x\Vert^{(0)} < r, \ \ M(u) \leq \left(\frac{C_3}{\alpha_0 - \alpha_*} + C(x)\right)(\alpha^* - \alpha_0)^{\gamma}(1+\Vert x \Vert_{\alpha_*})\right \}.
\end{align}
where $M(u)$ is given by
\[
 M(u) := \sup \limits_{ 0 \leq \tau \leq \frac{\alpha^* - \alpha_0}{\lambda}}\sup \limits_{\bfrac{0 \leq t < T(\alpha)}{\alpha \in (\alpha_0, \alpha^*]}}\ (\alpha - \alpha_0 - \lambda t)^{\gamma}\Vert B(u(t),\tau)\Vert_{\alpha}.
\] 
\begin{Lemma}\label{LEMMA:05}
 We have $U(\cdot,0)x \in S_x$.
\end{Lemma}
\begin{proof}
Using (A2) we obtain
\begin{align*}
 \| U(\cdot, 0)x\|^{(\gamma)} = \sup \limits_{\bfrac{0 \leq t < T(\alpha)}{\alpha \in (\alpha_0, \alpha^*]}}\| U(t,0)x \|_{\alpha}
 \leq \frac{C_1 \| x \|_{\alpha_*} }{ (\alpha_0 -  \alpha_*)^{\beta} } < \infty,
\end{align*} 
i.e. $U(\cdot,0)x \in S^{\gamma}$. 
By A3 we see that for $\alpha \in (\alpha_0, \alpha^*)$ and $0 \leq t < T(\alpha)$
\begin{align*}
 & \Vert U(t,0)x - x \Vert_{\alpha} \leq \Vert U(t,0)x - x\Vert_{\alpha_0} \leq C(x) t \leq C(x) \frac{\alpha - \alpha_0}{\lambda}
\end{align*}
and hence
\begin{align}\label{NONLINEAR:03}
 \Vert U(\cdot,0)x - x\Vert^{(0)} \leq C(x)\frac{\alpha^* - \alpha_0}{\lambda} \leq r \frac{\lambda_0}{\lambda}.
\end{align}
 Let us  show that, for all $0 \leq \tau \leq \frac{\alpha^* - \alpha_0}{\lambda}$, $\alpha \in (\alpha_0,\alpha^*]$ and all $0\leq t < T(\alpha)$, we have
\begin{align}\label{EQ:100}
 f(t) := \rho(t)^{\gamma}g(t) \leq \left(\frac{C_3}{\alpha_0 - \alpha_*}+C(x)\right)(\alpha^* - \alpha_0)^{\gamma}(1+\Vert x \Vert_{\alpha_*}),
\end{align}
where $\rho(t) := \alpha - \alpha_0 - \lambda t$ and $g(t) := \Vert B(U(t,0)x,\tau)\Vert_{\alpha}$. 
Conditions A1 and B1 imply that $f$ is continuous. Using
\[
 (Df)(t) = \underset{s \searrow t}{\limsup}\ \frac{f(s) - f(t)}{s-t}
\]
we obtain
\[
 (Df)(t) = - \gamma \lambda \rho(t)^{\gamma - 1}g(t) + \rho(t)^{\gamma }(Dg)(t).
\] 
Define $\alpha'$ by $\alpha = \alpha' + \frac{\rho(t)}{2}$, then $0 \leq t < T(\alpha') < T(\alpha)$ and, for $s \in (t,T(\alpha'))$,
\begin{align*}
 g(s) - g(t) &\leq \Vert B(U(s,0)x,\tau) - B(U(t,0)x,\tau) \Vert_{\alpha} \leq \frac{C_2}{(\alpha - \alpha')^{1 -\beta}} \Vert U(s,0)x - U(t,0)x\Vert_{\alpha'}.
\end{align*}
Dividing by $s - t$ and letting $s \searrow t$ we conclude from (A3)
\begin{align*}
 (Dg)(t) &\leq \frac{C_2C(x)}{(\alpha - \alpha')^{1-\beta}} = \frac{2^{1-\beta}C_2C(x)}{\rho(t)^{1-\beta}}.
\end{align*}
Using $\lambda > \lambda_0 > \frac{2^{1-\beta} C_2 (\alpha^* - \alpha_0)^{\beta }}{\gamma ( 1 + \| x \|_{\alpha_*})}$ we obtain 
\begin{align*}
 \rho(t)(Df)(t) &\leq - \gamma  \lambda f(t) + 2^{1-\beta}C_2C(x)\rho(t)^{\gamma+ \beta}
 \\ &\leq - \gamma \lambda f(t) + \gamma \lambda\left( \frac{C_3}{\alpha_0 - \alpha_*} + C(x)\right)(\alpha^* - \alpha_0)^{\gamma}(1+\Vert x \Vert_{\alpha_*})
\end{align*} 
and hence if $f(t) > ( \frac{C_3}{\alpha_0 - \alpha_*} + C(x))(\alpha^* - \alpha_0)^{\gamma}(1+\Vert x \Vert_{\alpha_*})$, then $(Df)(t) < 0$. 
By $x \in S_x$ we obtain
\[
 f(0) \leq \left( \frac{C_3}{\alpha_0 - \alpha_*} + C(x)\right)(\alpha^* - \alpha_0)^{\gamma}(1+\Vert x \Vert_{\alpha_*}).
\]
This implies \eqref{EQ:100} and hence
\begin{align}\label{NONLINEAR:11}
 M(U(\cdot,0)x) \leq \left( \frac{C_3}{\alpha_0 - \alpha_*} + C(x)\right)(\alpha^* - \alpha_0)^{\gamma}(1+\Vert x \Vert_{\alpha_*}).
\end{align}
\end{proof}
The next lemma is the only place in the proof of Theorem \ref{NONLINEARTH:00} where we have to distinguish between the cases $(\beta = 0, r \in (0,\infty])$ and 
$(\beta \in (0,\frac{1}{2}), r = \infty)$.
\begin{Lemma}\label{LEMMA:06}
 We have $U(\cdot,0)x + \mathcal{T}(u) \in S_x$ whenever $u \in S_x$.
\end{Lemma}
\begin{proof}
Take $u \in S_x$. Let us first show that $\| U(\cdot,0)x + \mathcal{T}(u) - x\|^{(0)} \leq r \frac{\lambda_0}{\lambda}$.
If $r = \infty$, then nothing has to be shown. Suppose that $r \in (0,\infty)$. In this case we necessarily have $\beta = 0$. Fix $\alpha \in (\alpha_0, \alpha^*]$ and $0 \leq t < T(\alpha)$. For $\alpha'$ defined by $\alpha = \alpha' + \frac{\rho(t)}{2}$ 
with $\rho(t) = \alpha - \alpha_0 - \lambda t$ we get $0 \leq t < T(\alpha') < T(\alpha)$.
Hence, by A2 and $u \in S_x$, we obtain 
\begin{align*}
 \Vert \mathcal{T}(u)(t)\Vert_{\alpha} &\leq \int \limits_{0}^{t}\Vert U(t,s)B(u(s),s)\Vert_{\alpha}\dm s 
 \\ &\leq C_1 \int \limits_{0}^{t}\Vert B(u(s),s)\Vert_{\alpha'}\dm s
 \\ &\leq C_1M(u)\int \limits_{0}^{t}(\alpha' - \alpha_0 - \lambda s)^{-\gamma}\dm s 
 \\ &\leq C_1M(u)\frac{(\alpha' - \alpha_0)^{1-\gamma }}{\lambda(1 - \gamma)}
 \\ &\leq \frac{C_1}{\lambda(1-\gamma)} \left(\frac{C_3}{\alpha_0 - \alpha_*} + C(x)\right)(\alpha^* - \alpha_0)(1+\Vert x \Vert_{\alpha_*})
\end{align*} 
which yields 
\begin{align*}
 \Vert \mathcal{T}(u)\Vert^{(0)} &\leq \frac{C_1}{(1-\gamma)\lambda} \left(\frac{C_3}{\alpha_0 - \alpha_*} + C(x)\right)(\alpha^* - \alpha_0)(1+\Vert x\Vert_{\alpha_*}).
\end{align*} 
Using \eqref{NONLINEAR:03} implies 
\begin{align}
 \notag & \Vert U(\cdot,0)x + \mathcal{T}(u) - x\Vert^{(0)} \leq \Vert U(\cdot,0)x - x\Vert^{(0)} + \Vert \mathcal{T}(u)\Vert^{(0)} 
 \\ \notag &\leq \frac{C(x)(\alpha^* - \alpha_0)}{\lambda} + \frac{C_1}{(1-\gamma)\lambda} \left(\frac{C_3}{\alpha_0 - \alpha_*} + C(x)\right)(\alpha^* - \alpha_0)(1+\Vert x\Vert_{\alpha_*})
 \\ &\leq r \frac{\lambda_0}{\lambda}. \label{REST1}
\end{align} 
For the second condition in the definition of $S_x$ we have to show that
\[
 M(U(\cdot,0)x + \mathcal{T}(u)) \leq \left(\frac{C_3}{\alpha_0 - \alpha_*} + C(x)\right)(\alpha^* - \alpha_0)^{\gamma}(1+\Vert x\Vert_{\alpha_*}).
\] 
Here we may treat both cases, $(\beta = 0, r \in (0,\infty])$ and $(\beta \in (0,\frac{1}{2}), r = \infty)$, simultaneously. 
Indeed, similarly to the proof of Lemma \ref{LEMMA:05}, it suffices to show that, for all $0 \leq \tau \leq T(\alpha^*)$, $\alpha \in (\alpha_0,\alpha^*]$ and 
$0 \leq t < T(\alpha)$, we have
\begin{align}\label{EQ:101}
 f(t) = \rho(t)^{\gamma}g(t) \leq \left(\frac{C_3}{\alpha_0 - \alpha_*}+C(x)\right)(\alpha^* - \alpha_0)^{\gamma}(1+\Vert x \Vert_{\alpha_*}),
\end{align}
where now $g(t) := \Vert B(U(t,0)x + \mathcal{T}(u)(t),\tau)\Vert_{\alpha}$. Again we obtain
\begin{align}\label{EQ:108}
 (Df)(t) = - \gamma \lambda \rho(t)^{\gamma - 1}g(t) + \rho(t)^{\gamma}(Dg)(t).
\end{align}
Recall $\alpha = \alpha' + \frac{\rho(t)}{2}$ and let $\alpha''$ be defined by $\alpha'' = \alpha' + \frac{\rho(t)}{4}$. 
Then $\alpha = \alpha'' + \frac{\rho(t)}{4}$ and
\[
 0 \leq t < T(\alpha') < T(\alpha'') < T(\alpha).
\]
Using A2, A3 and B2 we obtain, for $s \in (t, T(\alpha'))$,
\begin{align*}
 g(s) - g(t) 
 &\leq \frac{C_2}{(\alpha - \alpha'')^{1-\beta}}\Vert U(s,0)x - U(t,0)x\Vert_{\alpha''} + \frac{C_2}{(\alpha - \alpha'')^{1-\beta}}\Vert \mathcal{T}(u)(s) - \mathcal{T}(u)(t) \Vert_{\alpha''}
 \\ &\leq \frac{4^{1- \beta}C_2}{\rho(t)^{1 -\beta}}\Vert U(s,0)x - U(t,0)x\Vert_{\alpha'} + \frac{4^{1-\beta}C_2}{\rho(t)^{1-\beta}}\int \limits_{t}^{s}\Vert (U(s,\tau) - U(t,\tau))B(u(\tau),\tau)\Vert_{\alpha''}\dm \tau
 \\ &\leq \frac{4^{1- \beta}C_2}{\rho(t)^{1 -\beta}}C(x)|s-t| + \frac{4C_1C_2}{\rho(t)}\int \limits_{t}^{s}\Vert B(u(\tau),\tau)\Vert_{\alpha'}\dm \tau.
\end{align*}
Dividing by $s-t$, taking the limit $s \searrow t$ and finally using $u \in S_x$ gives
\begin{align*}
 (Dg)(t) &\leq 4^{1-\beta}C_2\rho(t)^{-1 +\beta}C(x) + \frac{4C_1C_2}{\rho(t)}\Vert B(u(t),t)\Vert_{\alpha'}
 \\ &\leq 4^{1-\beta}C_2C(x)\rho(t)^{-1 + \beta} + 4C_1C_2\rho(t)^{-1}M(u)(\alpha' - \alpha_0 - \lambda t)^{-\gamma}
 \\ &\leq 4^{1-\beta}C_2C(x)\rho(t)^{-1 + \beta} 
 \\ & \ \ \ + 2^{2+\gamma}C_1C_2\rho(t)^{-\gamma - 1}\left(\frac{C_3}{\alpha_0 - \alpha_*} + C(x)\right)(\alpha^* - \alpha_0)^{\gamma}(1+\Vert x\Vert_{\alpha_*}).
\end{align*}
Hence we obtain from \eqref{EQ:108} and the definition of $\lambda_0$
\begin{align*}
 \rho(t)(Df)(t) &\leq - \gamma \lambda f(t) + 4^{1-\beta}C_2C(x)\rho(t)^{\beta+\gamma} 
 \\ & \ \ + 2^{2+\gamma}C_1C_2\left(\frac{C_3}{\alpha_0 - \alpha_*} + C(x)\right)(\alpha^* - \alpha_0)^{\gamma}(1+\Vert x\Vert_{\alpha_*})
 \\ &\leq - \gamma \lambda f(t) + \gamma \lambda \left(\frac{C_3}{\alpha_0 - \alpha_*} + C(x)\right)(\alpha^* - \alpha_0)^{\gamma}(1+\Vert x \Vert_{\alpha_*}).
\end{align*}
The assertion follows by similar arguments to Lemma \ref{LEMMA:05}.
\end{proof}
We are now prepared to complete the proof of Theorem \ref{NONLINEARTH:00}.
\begin{proof}[Proof of Theorem \ref{NONLINEARTH:00}]
Define recursively a sequence $(u^{(k)})_{k \in \N_0}$ by setting $u^{(0)} = U(\cdot,0)x \in S_x$ and 
$u^{(k+1)} = U(\cdot,0)x + \mathcal{T}(u^{(k)})$ for $k \in \N_0$.
Lemma \ref{LEMMA:05} and Lemma \ref{LEMMA:06} imply $u^{(k)} \in S_x$ for all $k \in \N$ and from Lemma \ref{LEMMA:04} we obtain
\[
 \Vert u^{(k+1)} - u^{(k)} \Vert^{(\gamma)} \leq \left(\frac{\lambda_0}{\lambda}\right)^k\Vert u^{(1)} - u^{(0)}\Vert^{(\gamma)}.
\]
Hence $(u^{(k)})_{k \geq 0}$ is a Cauchy-sequence in $S^{\gamma}$ which has a limit $u = \lim_{k \to \infty} u^{(k)} \in S^{\gamma}$.
In view of \ref{NONLINEAR:03} and \eqref{REST1} we conclude that the limit $u$ satisfies $\| u - x \|^{(0)} \leq r$.
In particular, it is a solution to \eqref{NONLINEAR:00} which satisfies the desired estimate \eqref{EQ:106}. This proves the existence of a solution with the desired properties. Let $v$ be another solution with property \eqref{EQ:03}. 
Applying Lemma \ref{LEMMA:03} and Lemma \ref{LEMMA:04} to $v$ and $u_{\alpha_0, \lambda, \gamma}$ yields
\[
  \| v - u_{\alpha_0, \lambda, \gamma}\|^{(\gamma)}
  = \| \mathcal{T}(u) - \mathcal{T}(v) \|^{(\gamma)} \leq \frac{\lambda_0}{\lambda}\| u - v\|^{(\gamma)},
\]
which yields $v = u$ and hence completes the proof of Theorem \ref{NONLINEARTH:00}.
\end{proof}

\subsection{Existence of classical solutions}
Below we study existence of classical solutions to equation \eqref{NONLINEAR:08}. 
Therefore, let $(\E_{\alpha}, \vertiii{\cdot}_{\alpha})_{\alpha \in [\alpha_*, \alpha^*]}$ be another scale of Banach spaces with $\B_{\alpha} \subset \E_{\alpha}$
continuously embedded and $\vertiii{\cdot}_{\alpha} \leq \Vert \cdot \Vert_{\alpha}$ for all $\alpha \in [\alpha_*, \alpha^*]$. 
The next condition relates the evolution system $U(t,s)$ to its infinitesimal operator $A(t)$.
\begin{enumerate}
 \item[A4.] There exists a family of linear operators $(A(t))_{t \in [0,T)}$ such that, for all $\alpha' < \alpha$,
 \[
  \left [0, T\right) \ni t \longmapsto A(t) \in L(\B_{\alpha'}, \E_{\alpha})
 \]
 is strongly continuous. We have that $(t,s) \longmapsto U(t,s)y \in \E_{\alpha}$ is continuous for any $y \in \E_{\alpha'}$ and all $\alpha' < \alpha$.
 Moreover, for $y \in \B_{\alpha'}$ the function $(t,s) \longmapsto U(t,s)y$ is differentiable in $\E_{\alpha}$ and
 \begin{align}\label{EQ:103}
  \frac{\partial U(t,s)}{\partial t}y = A(t)U(t,s)y, \qquad \qquad \frac{\partial U(t,s)}{\partial s}y = - U(t,s)A(s)y
 \end{align}
 hold for $0 \leq s \leq t < T$. The case $s=t$ should be understood as right or left derivative correspondingly.
\end{enumerate}
Note that the equality in \eqref{EQ:103} is well-defined. Indeed, for given $\alpha' < \alpha$, $y \in \B_{\alpha'}$
and $0 \leq s \leq t < T$ take $\alpha'' \in (\alpha', \alpha)$, 
then $U(t,s)y \in \B_{\alpha'}$, $A(s)y \in \E_{\alpha''}$ and hence $A(t)U(t,s)y \in \E_{\alpha}$, $U(t,s)A(s)y \in \E_{\alpha}$.
\begin{Remark}
 If $\E_{\alpha} = \B_{\alpha}$, then conditions A1, A4 imply condition A3.
\end{Remark}
The definition of a classical solution is given below.
\begin{Definition}
 Take $\alpha_0 \in [\alpha_*, \alpha^*)$ and $T' \in (0,T]$. A classical $\B$-valued solution to \eqref{NONLINEAR:08} on the interval $[0. T')$ is a function $u: [0, T(\alpha^*, \alpha_0;T')) \longrightarrow \B_{\alpha^*}$ such that, for all $\alpha \in (\alpha_0, \alpha^*]$,
 the restriction
 \begin{align}\label{NONLINEAR:09}
  u|_{[0, T(\alpha, \alpha_0;T'))} \in C^1\left( [0, T(\alpha, \alpha_0;T'));\E_{\alpha}\right) \cap C\left( [0, T(\alpha, \alpha_0;T'));\B_{\alpha}\right) 
 \end{align}
 satisfies $u(t) \in B_r^{\alpha}(x)$ for all $t \in [0, [0, T(\alpha, \alpha_0;T'))$ and,
 it is a classical solution to \eqref{NONLINEAR:08} in $\E_{\alpha}$.
\end{Definition}
Concerning existence of classical solutions to \eqref{NONLINEAR:00} we obtain the following.
\begin{Theorem}\label{NONLINEARTH:01} 
 Assume that conditions A1 -- A4 and B1 -- B3 are satisfied
 either for $(\beta = 0, r \in (0,\infty])$ or $(\beta \in (0,\frac{1}{2}), r \in (0,\infty])$.
 Then, for each $\alpha_0 \in (\alpha_*, \alpha^*)$, $\gamma \in (\beta, 1- \beta)$ and $\lambda > \lambda_0(x,\alpha_0, \gamma,r,T)$, the solution $u_{\alpha_0,\lambda,\gamma}$ given by Theorem \ref{NONLINEARTH:00} is also a classical $\B$-valued solution to \eqref{NONLINEAR:08} on the interval $[0,\frac{\alpha^* - \alpha_0}{\lambda})$. 
\end{Theorem}
The proof of this result is a direct consequence of the following observation.
\begin{Lemma}\label{INTEGRAL}
 Assume that conditions A1 -- A4 and B1 -- B3 are satisfied
 either for $(\beta = 0, r \in (0,\infty])$ or $(\beta \in (0,\frac{1}{2}), r \in (0,\infty])$.
 Let $\alpha_0 \in [\alpha_*, \alpha^*)$, $x \in \B_{\alpha_*}$ and $u:[ 0, \frac{\alpha^*- \alpha_0}{\lambda}) \longrightarrow \B_{\alpha^*}$. 
 Then, $u$ is a classical $\B$-valued solution to \eqref{NONLINEAR:08} if and only 
 if $u$ is a solution to \eqref{NONLINEAR:00} in the scale $(\B_{\alpha})_{\alpha \in [\alpha_0,\alpha^*]}$.
\end{Lemma} 
\begin{proof}
 Suppose that $u$ is a classical $\B$-valued solution to \eqref{NONLINEAR:08}. 
 Then, for each $\alpha \in (\alpha_0, \alpha^*]$ and all $0 \leq s < t < \frac{\alpha - \alpha_0}{\lambda}$
 \begin{align}\label{EQ:104}
  \frac{\partial}{\partial s}\left(U(t,s)u(s)\right) = U(t,s)B(u(s),s)
 \end{align}
 holds in $\E_{\alpha}$. Indeed, take $\alpha', \alpha''$ with $\alpha_0 < \alpha' < \alpha'' < \alpha$ and let $\delta > 0$ be small enough 
 such that $0 \leq s \leq s+h < t < \frac{\alpha' - \alpha_0}{\lambda}$, for all $|h| < \delta$.
 Write
 \begin{align*}
  &\ \frac{U(t,s+h)u(s+h) - U(t,s)u(s)}{h} 
 \\ &= \frac{U(t,s+h) - U(t,s)}{h}u(s) + \left( U(t,s+h) - U(t,s) \right) \frac{u(s+h) - u(s)}{h} + U(t,s) \frac{u(s+h) - u(s)}{h}.
 \end{align*}
 The first term tends, by A4, to $-U(t,s)A(s)u(s)$ in $\E_{\alpha}$, where we have used $u(s) \in \B_{\alpha'}$ and $A(s)u(s) \in \E_{\alpha}$.
 For the second term observe that $K_s = \{ \frac{u(s+h) - u(s)}{h} \ | \ |h| \leq \delta \} \cup \{ \frac{\dm u(s)}{\dm s} \}$ is compact in $\E_{\alpha}$
 and hence
 \[
  \vertiii{ \left( U(t,s+h) - U(t,s) \right) \frac{u(s+h) - u(s)}{h}}_{\alpha} \leq \sup \limits_{z \in K_s} \vertiii{ U(t,s+h)z - U(t,s)z}_{\alpha} \to 0, \ \ h \to 0
 \]
 by the strong continuity of $U(t,s)$. Finally, using $\frac{u(s+h) - u(s)}{h} \longrightarrow A(s)u(s) + B(u(s),s)$ in $\E_{\alpha}$
 together with the boundedness of $U(t,s)$ on $\E_{\alpha}$, shows that the last term tends to $U(t,s)( A(s)u(s) + B(u(s),s))$ in $\E_{\alpha}$.
 Altogether, this implies \eqref{EQ:104}. Integrating \eqref{EQ:104} over $s \in [0,t]$ yields \eqref{NONLINEAR:00}. 

 For the converse let $v$ be such that $v \in C([0,T(\alpha)); \B_{\alpha})$ for all $\alpha \in (\alpha_0, \alpha^*]$.
 Fix $\alpha \in (\alpha_0, \alpha^*]$, $t \in [0,\frac{\alpha - \alpha_0}{\lambda})$ and let $\alpha' \in (\alpha_0, \alpha)$ such that 
 $0 \leq t < \frac{\alpha' - \alpha_0}{\lambda} < \frac{\alpha - \alpha_0}{\lambda}$ holds.
 Then, $v(s) \in \B_{\alpha'}$ for $s \in [0,t]$ and hence $(t,s) \longmapsto U(t,s)v(s) \in \B_{\alpha}$ is continuous.
 Moreover, for each fixed $s \in [0,t]$, $t \longmapsto U(t,s)v(s)$ is continuously differentiable in $\E_{\alpha}$. Thus
 \[
  \left [0,\frac{\alpha - \alpha_0}{\lambda}\right ) \ni t \longmapsto \int \limits_{0}^{t}U(t,s)v(s)\dm s
 \]
 is continuous in $\B_{\alpha}$ and continuously differentiable in $\E_{\alpha}$ satisfying
 \[
  \frac{\dm}{\dm t} \int \limits_{0}^{t}U(t,s)v(s)\dm s = v(t) + A(t)\int \limits_{0}^{t}U(t,s)v(s)\dm s.
 \]
 Let $u$ solve \eqref{NONLINEAR:00}. 
 Applying above argumentation to $v(s) := B(u(s),s)$ where $u$ solves \eqref{NONLINEAR:00}, shows that $u(t)$ is differentiable in $\E_{\alpha}$.
 Hence differentiating \eqref{NONLINEAR:00} yields \eqref{NONLINEAR:08}. 
\end{proof}

\section{Stability with respect to parameters}
For the whole section we suppose that the conditions below are satisfied.
\begin{enumerate}
 \item[C1.] There exist $x_n,x \in \B_{\alpha_*}$ with $x_n \longrightarrow x$ as $n \to \infty$.
 \item[C2.] There exist linear operators $(U(s,t))_{0 \leq s \leq t < T}$ and 
 $(U_n(t,s))_{0 \leq s \leq t < T}$ for $n \in \N$, 
 satisfying properties A1 and A2 with constants $C_1 > 0$ and $\beta \in [0,\frac{1}{2})$ independent of $n \in \N$.
 \item[C3.] For any $\alpha \in [\alpha_*, \alpha^*]$ there exists a constant $C(\alpha)  > 0$ such that for all $0 \leq s,t < T$
 \[
  \Vert U_n(t,0)x_n - U_n(s,0)x_n\Vert_{\alpha} \leq C(\alpha)|t-s|, \ \ n \in \N
 \]
 holds.
 \item[C4.] There exist operators $B$ and $B_n$ satisfying properties B1 -- B3 with constants 
 \\$\lambda, r, C_2, C_3 > 0$ independent of $n \in \N$.
 \item[C5.] For all $\alpha', \alpha \in [\alpha_*, \alpha^*]$ with $\alpha' < \alpha$ and each $z \in \B_{\alpha'}$ we have
 \begin{align}\label{NONLINEAR:17}
  U_n(t,s)z \longrightarrow U(t,s)z, \ \ n \to \infty
 \end{align}
 in $\B_{\alpha}$ uniformly on compacts in $(t,s)$. If in addition $\Vert z - x \Vert_{\alpha'} \leq r$, then we have
 \[
  B_n(z,t) \longrightarrow B(z,t), \ \ n \to \infty
 \]
 in $\B_{\alpha}$ uniformly on compacts in $t$.
\end{enumerate}
 By $x_n \longrightarrow x$ one finds that $\| x_n\|_{\alpha_*}$ is bounded in $n$ and hence using the particular form of the constants $\lambda_0$ from previous section we may define 
 \[
  \lambda_1(\alpha_0,\gamma,r,T) = \max \left\{ \lambda_0(x, \alpha_0, \gamma,r,T),\ \sup \limits_{n \in \N}\lambda_0(x_n, \alpha_0,\gamma, r,T) \right \} \in (0,\infty).
 \]
The following is our main result for the stability of solutions to \eqref{NONLINEAR:00}
\begin{Theorem}\label{NONLINEARTH:02} 
 Suppose that conditions C1 -- C5 are satisfied either for $(\beta = 0,r \in (0,\infty])$ or $(\beta \in (0,\frac{1}{2}), r \in (0,\infty])$. Let $\alpha_0 \in (\alpha_*, \alpha^*)$, $\gamma \in (\beta, 1 - \beta)$ and $\lambda > \lambda_1(\alpha_0, \gamma,r,T)$. Denote by $u = u_{\alpha_0, \lambda, \gamma}$ the solution to \eqref{NONLINEAR:00} and by $u_n = u_{n, \alpha_0, \lambda, \gamma}$ the solutions to 
 \begin{align}\label{NONLINEAR:12}
  u_n(t) = U_n(t,0)x_n + \int \limits_{0}^{t}U_n(t,s)B_n(u_n(s),s)\dm s
 \end{align}
 on the interval $[0, \frac{\alpha^* - \alpha_0}{\lambda})$ in the scale $(\B_{\alpha})_{\alpha \in [\alpha_0, \alpha^*]}$.
 Then, for each $\alpha \in (\alpha_0, \alpha^*]$ and $T' \in (0, \frac{\alpha - \alpha_0}{\lambda})$, one has
 \begin{align}\label{NONLINEAR:13}
  \sup \limits_{t \in [0,T']} \| u_n(t) - u(t) \|_{\alpha} \longrightarrow 0, \ \ n \to \infty.
 \end{align} 
\end{Theorem}
 The rest of this section is devoted to the proof.
 Since \eqref{NONLINEAR:17} implies that (A3) also holds for $U(t,s)$ we find that $u$ and $u_n$ are, indeed, well-defined.
 Denote by $\mathcal{T}_n$ the nonlinear integral operator given as in \eqref{NONLINEAR:01} with $B$ and $U$ replaced by $B_n$ and $U_n$ and recall that $T(\alpha) = \frac{\alpha - \alpha_0}{\lambda}$.
 \begin{Lemma}\label{NONLINEARLEMMA:02}
  For each $\gamma \in (\beta, 1 - \beta)$, $\Vert \mathcal{T}_n(U_n(\cdot,0)x_n)\Vert^{(\gamma)}$ is uniformly bounded in $n\in \N$.
 \end{Lemma}
 \begin{proof}
  An analogous estimate to \eqref{NONLINEAR:11} shows that 
  \[
   \sup \limits_{n \in \N} \sup \limits_{0 \leq \tau \leq  \frac{\alpha^* - \alpha_0}{\lambda}} \ \sup \limits_{\bfrac{0 \leq t < T(\alpha)}{\alpha \in (\alpha_0, \alpha^*]}}(\alpha - \alpha_0 - \lambda t)^{\gamma} \Vert B_n(U_n(\cdot,0)x_n,\tau)\Vert_{\alpha} < \infty.
  \]
  Let $\alpha \in (\alpha_0, \alpha^*]$, $t \in [0, T(\alpha))$, $\rho(t) = \alpha - \alpha_0 - \lambda t$ and define $\alpha'$ by the relation
  $\alpha = \alpha' + \frac{\rho(t)}{2}$. Then $0 \leq t < \frac{\alpha' - \alpha_0}{\lambda}$ and the assertion follows from
  \begin{align*} 
   & (\alpha - \alpha_0 - \lambda t)^{\gamma}\left \Vert \int \limits_{0}^{t} U_n(t,s)B_n(U_n(s,0)x_n,s)\dm s\right\Vert_{\alpha} 
   \\ &\leq C_1 \frac{(\alpha - \alpha_0 - \lambda t)^{\gamma}}{(\alpha - \alpha')^{\beta}}\int \limits_{0}^{t}\Vert B_n(U_n(s,0)x_n,s)\Vert_{\alpha'}\dm s
   \\ &\leq C C_1 2^{\beta}\rho(t)^{\gamma - \beta} \int \limits_{0}^{t}(\alpha' - \alpha_0 - \lambda s)^{-\gamma}\dm s
   \\  &\leq \frac{2^{\beta}C C_1 \rho(t)^{\gamma - \beta}(\alpha^* - \alpha_0)^{1 - \gamma}}{\lambda (1-\gamma)}
   \leq \frac{2^{\beta}C C_1}{\lambda (1-\gamma)} (\alpha^* - \alpha_0)^{1-\beta}. 
  \end{align*}
 \end{proof}
 Denote by $(u^{(k)})_{k \in \N}$ the sequence defined by 
 $u^{(0)} = U(\cdot,0)x$, $u^{(k+1)} = U(\cdot,0)x + \mathcal{T}(u^{(k)})$. Similarly let $(u_n^{(k)})_{k \in \N}$ be given by $u_n^{(0)} = U_n(\cdot,0)x_n$ and
 $u_n^{(k+1)} = U_n(\cdot,0)x_n + \mathcal{T}_n(u_n^{(k)})$. For $\gamma \in (\beta, 1-\beta)$ we obtain
 \begin{align*}
  \Vert u^{(k)} - u\Vert^{(\gamma)} &\leq \sum \limits_{j=k}^{\infty}\Vert u^{(j+1)} - u^{(j)}\Vert^{(\gamma)} \leq \sum \limits_{j=k}^{\infty}\left(\frac{\lambda_0}{\lambda}\right)^j \Vert u^{(1)} - u^{(0)}\Vert^{(\gamma)}
  \\ &= \sum \limits_{j=k}^{\infty}\left(\frac{\lambda_0}{\lambda}\right)^j \Vert \mathcal{T}(U(\cdot,0)x)\Vert^{(\gamma)}
 \end{align*}
 and similarly, by Lemma \ref{NONLINEARLEMMA:02},
 \[
  \Vert u^{(k)}_n - u_n\Vert^{(\gamma)} 
  \leq \sum \limits_{j=k}^{\infty}\left( \frac{\lambda_0}{\lambda}\right)^j  \sup \limits_{n \in \N} \Vert \mathcal{T}_n(U_n(\cdot, 0)x_n)\Vert^{(\gamma)} < \infty.
 \]
 For $\alpha \in (\alpha_0, \alpha^*]$, write
 \begin{align*}
 \Vert u(t) - u_n(t)\Vert_{\alpha} &\leq (\alpha - \alpha_0 - \lambda t)^{-\gamma}\left(\Vert u^{(k)} - u\Vert^{(\gamma)} + \Vert u_n^{(k)} - u_n\Vert^{(\gamma)}\right)
 + \Vert u^{(k)}(t) - u_n^{(k)}(t)\Vert_{\alpha}
 \end{align*}
 and observe that the first two terms tend to zero uniformly in $n$ as $k \to \infty$. 
 Thus it suffices to show that for each $k$ and each $T' \in (0,T(\alpha))$
 \[
  \sup \limits_{t \in [0,T']} \Vert u^{(k)}(t) - u_n^{(k)}(t)\Vert_{\alpha} \to 0, \ \ n \to \infty.
 \]
 However, this is a consequence of the following lemma.
 \begin{Lemma}
  Let $v_n, v: [0, \frac{\alpha^* - \alpha_0}{\lambda}) \longrightarrow \B_{\alpha^*}$ be two functions such that for each $\alpha \in (\alpha_0, \alpha^*]$
  and all $n \in \N$ the conditions below are satisfied:
  \begin{enumerate}
   \item $\Vert v_n(t) - x_n \Vert_{\alpha} < r$ and $\Vert v(t) - x \Vert_{\alpha} < r$ for all $0 \leq t < \frac{\alpha - \alpha_0}{\lambda}$.
   \item $v_n|_{[0, \frac{\alpha - \alpha_0}{\lambda})}, v|_{[0, \frac{\alpha - \alpha_0}{\lambda})} \in C([0, \frac{\alpha - \alpha_0}{\lambda}); \B_{\alpha})$
   and $v_n(0) = x_n$, $v(0) = x$.
   \item For each $T' \in (0, \frac{\alpha - \alpha_0}{\lambda})$ one has
   \[
    \sup \limits_{t \in [0,T']} \Vert v_n(t) - v(t)\Vert_{\alpha} \to 0, \ \ n \to \infty.
   \]
  \end{enumerate}
  Then, for each $\alpha \in (\alpha_0, \alpha^*]$ and $T' \in (0, \frac{\alpha - \alpha_0}{\lambda})$, one has
  \[
   \sup \limits_{t \in [0,T]} \Vert U_n(t,0)x_n + \mathcal{T}_n(v_n)(t) - U(t,0)x - \mathcal{T}(v)(t)\Vert_{\alpha} \longrightarrow 0, \ \ n \to 0.
  \]
 \end{Lemma}
 \begin{proof}
 Write $\Vert U_n(t,0)x_n + \mathcal{T}_n(v_n)(t) - U(t,0)x - \mathcal{T}(v)(t)\Vert_{\alpha}  \leq I_1 + I_2$ where
 \begin{align*}
  I_1 = \| U_n(t,0)x_n - U(t,0)x \|_{\alpha}, \qquad I_2 = \| \mathcal{T}_n(v_n)(t) - \mathcal{T}(v)(t)\|_{\alpha}.
 \end{align*}
 Take $\e > 0$, then we find $n_0 \in \N$ such that, for all $n \geq n_0$ and $t \in [0,T']$,
 \begin{align*}
  I_1 &\leq  \Vert U_n(t,s)(x_n - x)\Vert_{\alpha} + \Vert (U_n(t,0) - U(t,0))x\Vert_{\alpha} \leq \e,
 \end{align*}
 where we have used C1, C2 and C5. For the second term we obtain
 \begin{align}\label{EQ:102}
  I_2 &\leq \| \mathcal{T}_n(v_n)(t) - \mathcal{T}_n(v)(t) \Vert_{\alpha} +  \Vert \mathcal{T}_n(v)(t) - \mathcal{T}(v)(t)\Vert_{\alpha}.
 \end{align}
 Take $\alpha', \alpha'' \in [\alpha_0, \alpha^*]$ such that
 $\alpha_0 < \alpha' < \alpha'' < \alpha$.
 Using C2 and C4 we find $n_1 \geq n_0$ such that, for all $n \geq n_1$, 
 \begin{align*}
  \Vert \mathcal{T}_n(v_n)(t) - \mathcal{T}_n(v)(t)\Vert_{\alpha} &\leq \int \limits_{0}^{t}\Vert U_n(t,s)(B_n(v_n(s),s) - B_n(v(s),s)\Vert_{\alpha}\dm s
  \\ &\leq \frac{C_1}{(\alpha - \alpha'')^{\beta}}\frac{C_2}{(\alpha'' - \alpha')^{1-\beta}}\int \limits_{0}^{t} \Vert v_n(s) - v(s)\Vert_{\alpha'}\dm s
  \\ &=  \frac{C_1}{(\alpha - \alpha'')^{\beta}}\frac{C_2}{(\alpha'' - \alpha')^{1-\beta}}T' \underset{s \in [0,T']}{\sup}\ \Vert v_n(s) - v(s)\Vert_{\alpha'} \leq \e.
 \end{align*}
 For the other term in \eqref{EQ:102} we obtain $\Vert \mathcal{T}_n(v)(t) - \mathcal{T}(v)(t)\Vert_{\alpha} \leq J_1 + J_2$, where
 \begin{align*}
  J_1 &= \int \limits_{0}^{t}\Vert U_n(t,s)(B_n(v(s),s) - B(v(s),s) )\Vert_{\alpha}\dm s,
\\ J_2 &= \int \limits_{0}^{t}\Vert (U_n(t,s) - U(t,s))B(v(s),s)\Vert_{\alpha}\dm s.
 \end{align*}
 Take $\alpha',\alpha''$ with $\alpha_0 < \alpha' < \alpha'' < \alpha$ such that $T' < \frac{\alpha' - \alpha_0}{\lambda}$.
 Then, $v_n(t),v(t) \in \B_{\alpha'}$ are continuous in $t \in [0,T']$ and hence
 the set $K_{T'} = \{ B(v(s),s)\ |\ s \in [0,T']\} \subset \B_{\alpha''}$ is compact. For large $n$, i.e. for $n \geq n_2 \geq n_1$, 
 we obtain by C5
 \[
  J_2 \leq \int \limits_{0}^{t}\underset{z \in K_{T'}}{\sup}\ \Vert (U_n(t,s) - U(t,s))z\Vert_{\alpha}\dm s < \e.
 \]
 Similarly the set $K_{T'}' = \{ v(s)\ |\ s \in [0,T']\} \subset \B_{\alpha'}$ is compact and there exists $\delta \in (0,1)$ such that
 $\Vert v(t) - x\Vert_{\alpha'} < (1-\delta)r$ for $t \in [0,T']$. Let $n_3 \geq n_2$ be such that $\Vert x - x_n \Vert_{\alpha_*} < \delta r$ holds for $n \geq n_3$. 
 Then for each $n \geq n_3$
 \[
  \Vert v(t) - x_n \Vert_{\alpha'} \leq \Vert v(t) - x \Vert_{\alpha'} + \Vert x - x_n \Vert_{\alpha'} < (1-\delta)r + \delta r = r.
 \]
 Hence, by C2 and C5, there exists $n_4 \geq n_3$ such that for all $n \geq n_4$
 \begin{align*}
  J_1 &\leq \frac{C_1}{(\alpha - \alpha'')^{\beta}}\int \limits_{0}^{t}\Vert B_n(v(s),s) - B(v(s),s)\Vert_{\alpha''}\dm s
  \\ &\leq \frac{C_1}{(\alpha - \alpha'')^{\beta}}\int \limits_{0}^{t}\underset{ z \in K_{T'}'}{\sup}\ \Vert B_n(z,s) - B(z,s)\Vert_{\alpha''}\dm s
  \leq \e.
 \end{align*}
 This proves the assertion.
 \end{proof}

\section{Epistatic mutation-selection balance model}

Let $X$ be a complete, locally compact metric space and $\sigma$ be a $\sigma$-finite, non-atomic Borel measure on $X$.
Elements of $X$ describe potential mutations and the value $\sigma(A)$ is the rate at which spontaneously a mutant allele arises from $A \subset X$. 
Such allele is characterized by its value $x \in A$. The collection of all mutant alleles is the so-called genotype $\gamma$. 
It is, by definition, a locally finite subset of $X$. The configuration space of all genotypes
\[
 \Gamma = \{ \gamma \subset X \ | \ |\gamma \cap X| < \infty \ \text{ for all compacts } K \subset X \}
\]
is equipped with the vague topology, i.e. the smallest topology such that $\gamma \longmapsto \sum_{x \in \gamma}f(x)$
is continuous for any continuous function $f$ having compact support. Then $\Gamma$ is a Polish space (see \cite{KK06}).

We assign to each genotype $\gamma \in \Gamma$ a \textit{selection cost} functional
\[
 \Phi(t,\gamma) = \sum \limits_{x \in \gamma} h(t,x) + \frac{1}{2}\sum \limits_{x \in \gamma}\sum \limits_{y \in \gamma \backslash x}\psi(t,x,y),
\]
where $h,\psi$ are non-negative, measurable functions with $\psi(t,x,y) = \psi(t,y,x)$, for all $x,y \in X$ and $t \geq 0$.
We suppose from now on that the following two conditions are fulfilled:
\begin{enumerate}
 \item[(D1)] $h \in C(\R_+; L^1(X, \sigma) \cap L^{\infty}(X,\sigma))$ and $\psi \in C(\R_+; L^1(X^2, \sigma^{\otimes 2}) \cap L^{\infty}(X^2, \sigma^{\otimes 2}))$.
 \item[(D2)] For any $T > 0$ we have
 \[
  \sup \limits_{(t,x) \in [0,T] \times X} \int \limits_{X}\psi(t,x,y)\dm \sigma(y) < \infty.
 \]
 \item[(D3)] $0 \leq a \in C_b(\R_+; L^{\infty}(X,\sigma))$.
\end{enumerate}

\subsection{The nonlinear Fokker-Planck equation}
Borel probability measures $\mu$ on $\Gamma$ describe the distribution of potential mutations. We call such measures states of the dynamics.
The associated dynamics of potential mutations is described by a family of states $(\mu_t)_{t \in [0,T)}$, which shall satisfy for a suitable 
collection of functions $F: \Gamma \longrightarrow \R$ the nonlinear Fokker-Planck equation \eqref{INTRO:01}.

Following \cite{FKKO07}, we give now a rigorous formulation to equation \eqref{INTRO:01}.
Set $\Gamma_0 = \{ \eta \subset X \ | \ |\eta| < \infty\}$ and denote by $B_{bs}(\Gamma_0)$ the space of all bounded functions $G$ 
such that there exists a compact $\Lambda = \Lambda(G) \subset X$ and $N = N(G) \in \N$ with $G(\eta) = 0$, whenever $|\eta| > N$ or $\eta \cap \Lambda^c \neq \emptyset$. Define a measure $\lambda$ on $\Gamma_0$ by
\[
 \int \limits_{\Gamma_0}G(\eta)\dm \lambda(\eta) = G( \{\emptyset \} ) + \sum \limits_{n=1}^{\infty}\frac{1}{n!}\int \limits_{X^n}G(\{x_1,\dots, x_n\})\dm \sigma^{\otimes}(x_1,\dots,x_n),
\]
where $\sigma^{\otimes n}$ is the product measure on $X^n$ and $G \in B_{bs}(\Gamma_0)$.
This definition clearly extends to all measurable functions $G$ such that one side of the equality is finite for $|G|$.

Let $F \in \mathcal{FP}(\Gamma)$ iff there exists $G \in B_{bs}(\Gamma_0)$ such that
\begin{align}\label{EQ:00}
 F(\gamma) = \sum \limits_{\bfrac{\eta \subset \gamma}{|\eta| < \infty}} G(\eta), \ \ \gamma \in \Gamma.
\end{align}
Note that for each $F \in \mathcal{FP}(\Gamma)$ there exists a compact $\Lambda \subset \R^d$, $N \in \N$ and $A > 0$ such that 
$F(\gamma) = F(\gamma \cap \Lambda)$ and $|F(\gamma)| \leq A( 1 + |\gamma \cap \Lambda|)^N$ hold for all $\gamma \in \Gamma$.
From this it is easily seen that the integral in \eqref{EQ:105} is well-defined for any $F \in \mathcal{FP}(\Gamma)$ and any $\gamma \in \Gamma$.
For the remaining terms in \eqref{EQ:105} additional conditions on the state evolution $(\mu_t)_{t \in [0,T)}$ has to be imposed
which is summarized in the following definition.
\begin{Definition}
 Fix $T > 0$.
 A family of probability measures $(\mu_t)_{t \in [0,T)}$ on $\Gamma$
 is a weak solution to \eqref{INTRO:01} if for any $F \in \mathcal{FP}(\Gamma)$ the conditions below are satisfied:
 \begin{enumerate}
  \item[(a)] $F, L(t,\mu_t)F, \Phi(t,\cdot) \in L^1(\Gamma, \dm \mu_t)$ for all $t \in [0,T)$.
  \item[(b)] $[0,T) \ni t \longmapsto \langle L(t,\mu_t)F, \mu_t \rangle$ is locally integrable.
  \item[(c)] $(\mu_t)_{t \in [0,T)}$ satisfies
  \[
   \langle F, \mu_t \rangle = \langle F, \mu_0 \rangle + \int \limits_{0}^{t}\langle L(t,\mu_t)F, \mu_t \rangle \dm t, \ \ t \in [0,T).
  \]
 \end{enumerate}
\end{Definition}

\subsection{Evolution of correlation functions}
Since \eqref{INTRO:01} is a nonlinear Fokker-Planck equation over an infinite dimensional phase space $\Gamma$,
it is unnatural to expect that uniqueness holds in the class of all probability measures on $\Gamma$.
Below we study, following \cite{FKKO07}, uniqueness to \eqref{INTRO:01} in a suitable subclass of states.
Namely, let $\mu$ be a probability measure on $\Gamma$ with finite local moments, i.e. $\int_{\Gamma}|\gamma \cap \Lambda|^n d\mu(\gamma) < \infty$ 
for all compacts $\Lambda \subset X$ and all $n \geq 1$.
The correlation function $k_{\mu}: \Gamma_0 \longrightarrow \R_+$ is defined, provided it exists, as the unique function satisfying
\begin{align}\label{EQ:01}
 \int \limits_{\Gamma}\sum \limits_{\bfrac{\eta \subset \gamma}{|\eta| < \infty}}G(\eta) \dm \mu(\gamma) 
 = \int \limits_{\Gamma_0}G(\eta)k_{\mu}(\eta)\dm \lambda_{\sigma}(\eta), \ \ G \in B_{bs}(\Gamma_0).
\end{align}
For additional references on the notion of correlation functions see \cite{KK02} and the references therein.
\begin{Remark}
 The Poisson measure $\pi_{z\sigma}$ with intensity measure $z d\sigma$, $z > 0$, is
 uniquely determined by $\pi_{z\sigma}(\{ \gamma \ | \ |\gamma \cap \Lambda| = n \}) = \frac{z^n \sigma(\Lambda)^n}{n!} e^{-z\sigma(\Lambda)}$, where $\Lambda \subset X$ is compact.
 It can be shown that it has correlation function $k_{\pi_{z\sigma}}(\eta) = z^{|\eta|}$.
\end{Remark}
Let $\alpha \geq 0$. We study \eqref{INTRO:01} in the class of states $\mathcal{P}_{\alpha}$,
where $\mu \in \mathcal{P}_{\alpha}$
iff it has finite local moments, its correlation function $k_{\mu}$ exists and satisfies for some constant $A_{\mu} > 0$ the Ruelle bound
\begin{align}\label{NONLINEAR:19}
 k_{\mu}(\eta) \leq A_{\mu} e^{\alpha|\eta|}, \ \ \eta \in \Gamma_0.
\end{align}
Let $\K_{\alpha}$ be the Banach space of all equivalence classes of functions $k$ with finite norm
\[
 \Vert k \Vert_{\K_{\alpha}} = \esssup \limits_{\eta \in \Gamma_0} |k(\eta)|e^{-\alpha |\eta|}.
\]
Clearly any $k \in \K_{\alpha}$ satisfies the Ruelle bound \eqref{NONLINEAR:19}.
\begin{Remark}
 There exists a one-to-one correspondence between elements in $\mathcal{P}_{\alpha}$ and positive definite functions in $\mathcal{K}_{\alpha}$,
 see \cite{KK02} and the references therein.
\end{Remark}
It is natural to study \eqref{INTRO:01} now in terms of correlation
functions $k_{\mu}$.
For this purpose we define a new nonlinear mapping $L^{\Delta}(t,\cdot)$ by the identity
\begin{align}\label{EQ:02}
 \int \limits_{\Gamma} L(t,\mu)F(\gamma) d\mu(\gamma)
 = \int \limits_{\Gamma_0}G(\eta)L^{\Delta}(t,k_{\mu})(\eta)\dm \lambda(\eta),
\end{align}
where $F$ is given by \eqref{EQ:00}.
Analogous to the calculations in \cite{FKKO07}, one can show that
\begin{align*}
 L^{\Delta}(t,k) &= - A_0^{\Delta}(t)k + A_1^{\Delta}(t)k + B^{\Delta}(t,k)k.
 \\ A_0^{\Delta}(t)k(\eta) &= \Phi(t,\eta)k(\eta) + \int \limits_{X}h(t,x)k(\eta \cup \{x\})\sigma(\dm x) 
 \\ & \ \ \ + \frac{1}{2}\int \limits_{X}\int \limits_{X}\psi(t,x,y)k(\eta \cup \{x\} \cup \{y\})\sigma(\dm x) \sigma(\dm y),
 \\ A_1^{\Delta}(t)k(\eta) &= - \sum \limits_{x \in \eta}\int \limits_{X}\psi(t,x,y)k(\eta \cup \{y\})\sigma(\dm y) + \sum \limits_{x \in \eta}a(t,x)k(\eta \backslash \{x\}),
 \\ B^{\Delta}(t,k) &= \int \limits_{X}h(t,x)k^{(1)}(x)\sigma(\dm x) + \frac{1}{2}\int \limits_{X} \int \limits_{X}\psi(t,x,y)k^{(2)}(x,y)\sigma(\dm x) \sigma(\dm y),
\end{align*}
where $B^{\Delta}(t,k)$ acts by multiplication.
The next lemma shows that $L^{\Delta}(t,\cdot)$ acts as a bounded linear operator in the scale $\K_{\alpha}$.
\begin{Lemma}\label{LEMMA:01}
 Let $\alpha > \alpha' \geq 0$ and fix any $k \in \K_{\alpha'}$. Then
 \begin{align*}
  \| A_0^{\Delta}(t)k \|_{\K_{\alpha}} &\leq \left( \frac{ \| h \|_{\infty} }{e(\alpha - \alpha')} + \frac{\| \psi \|_{\infty}}{e^2 (\alpha - \alpha')}\right) \| k \|_{\K_{\alpha'}}
  \\ & \ \ \ + \left( e^{\alpha'}\sup \limits_{t \in [0,T]}\int_{X}h(t,x)\sigma(dx) + \frac{e^{2\alpha'}}{2}\sup \limits_{t \in [0,T]}\int \limits_{X^2}\psi(t,x,y)\sigma(dx)\sigma(dy)\right)\| k \|_{\K_{\alpha'}}
  \\ \| A_1^{\Delta}(t)k \|_{\K_{\alpha}} &\leq \left( \frac{e^{\alpha'}}{e(\alpha - \alpha')} \sup \limits_{ (t,x) \in [0,T]\times X} \int \limits_{X}\psi(t,x,y)\sigma(dy) + \frac{e^{- \alpha'}}{e(\alpha - \alpha')}\| a\|_{\infty} \right) \| k \|_{\alpha'}
  \\ |B^{\Delta}(t,k)| &\leq \left(  e^{\alpha'} \sup \limits_{t \in [0,T]} \int \limits_{X}h(t,x)\sigma(dx) + \frac{e^{2 \alpha'}}{2}\sup \limits_{t \in [0,T]}\int \limits_{X^2}\psi(t,x,y)\sigma(dx)\sigma(dy) \right) \| k \|_{\K_{\alpha'}}.
 \end{align*}
\end{Lemma}
\begin{proof}
 For the first term in $A_0^{\Delta}(t)$ we use
 $\Phi(t,\eta) \leq |\eta| \| h\|_{\infty} + | \eta |^2 \| \psi \|_{\infty}$ combined with 
 $|k(\eta)| \leq \| k \|_{\K_{\alpha'}}e^{\alpha'|\eta|}
 \leq \| k \|_{\K_{\alpha'}} e^{- (\alpha - \alpha')|\eta|} e^{- \alpha|\eta|}$ and
 \[
  x^a e^{-b x} \leq \left( \frac{a}{b} \right)^{b} e^{-b}, \ \ x \geq 0, \ \ a,b > 0.
 \]
 For the second term we use 
 $|k(\eta \cup \{x\})| \leq e^{\alpha'}\| k \|_{\K_{\alpha'}} e^{\alpha'|\eta|} \leq e^{\alpha'}\| k \|_{\K_{\alpha'}}e^{\alpha|\eta|}$
 to obtain 
 \[
  \left| \int \limits_{X} h(t,x)k(\eta \cup \{x\}) \sigma(\dm x) \right| 
  \leq e^{\alpha'}\| k \|_{\K_{\alpha'}} \sup \limits_{t \in [0,T]} \int \limits_{X}h(t,x)\sigma(\dm x) e^{\alpha|\eta|}.
 \]
 The other terms can be estimated in the same way, see, e.g., \cite{FK16} for similar estimates.
\end{proof}
The next lemma shows that \eqref{INTRO:01} can be reformulated in terms of 
an evolution equation of correlation functions given by the nonlinear operator $L^{\Delta}(t,\cdot)$.
\begin{Lemma}\label{LEMMA:00}
 Let $(\mu_t)_{t \in [0,T)} \subset \mathcal{P}_{\alpha}$ and denote by $k_{t}$ the correlation function for $\mu_t$, $t \in [0,T)$.
 Then $(\mu_t)_{t \in [0,T)}$ is a weak solution to \eqref{INTRO:01} if and only if $(k_t)_{t \in [0,T)}$ satisfies:
 \begin{enumerate}
  \item[(i)] For any $G \in B_{bs}(\Gamma_0)$
  \[
   [0,T) \ni t \longmapsto \int \limits_{\Gamma_0}G(\eta)L^{\Delta}(t,k_{t})(\eta)d\lambda(\eta)
  \]
  is locally integrable.
  \item[(ii)] For any $G \in B_{bs}(\Gamma_0)$ and $t \in [0,T)$
  \begin{align}\label{NONLINEAR:16}
   \int \limits_{\Gamma_0}G(\eta)k_t(\eta)\dm \lambda_{\sigma}(\eta) 
   = \int \limits_{\Gamma_0}G(\eta)k_0(\eta)\dm \lambda_{\sigma}(\eta) + \int \limits_{0}^{t}\int \limits_{\Gamma_0}G(\eta)L^{\Delta}(s,k_s)(\eta)\dm \lambda_{\sigma}(\eta)\dm s.
  \end{align}
 \end{enumerate}
\end{Lemma}
The proof is based on the identities \eqref{EQ:01}, \eqref{EQ:02}
and can be obtained by similar arguments to \cite{FK16} and \cite{F17} 
where linear Fokker-Planck equations have been studied.
In order to keep this work self-contained a proof is given in the appendix.
The next statement shows the existence and uniqueness of strong solutions to \eqref{NONLINEAR:16}.
\begin{Theorem}
 Take $0 \leq \alpha_* < \alpha_0 < \alpha^*$.
 Then for each $k_0 \in \K_{\alpha_*}$ there exists $\lambda_1(k_0, \alpha^*, \alpha_*, \alpha_0) = \lambda_1 > 0$ such that for each $\lambda > \lambda_1$ there exists a unique classical $\K$-valued solution to 
 \begin{align}\label{EQ:120}
  \frac{\partial k_t}{\partial t} = L^{\Delta}(t,k_t), \ \ k_t|_{t=0} = k_0, \qquad 0 \leq t < \frac{\alpha^* - \alpha_0}{\lambda}
 \end{align}
 in the scale $\K = (\K_{\alpha})_{\alpha \in [\alpha_0, \alpha^*]}$.
\end{Theorem}
\begin{proof}
 Since $A_0^{\Delta}(t)$ is a sum of a multiplication operator and a bounded operator, 
 it is not difficult to see that, for any $\alpha \geq 0$, there exists a unique evolution system $(U_{\alpha}(t,s))_{0 \leq s \leq t} \subset L(\K_{\alpha})$ with
 the properties:
 \begin{enumerate}
  \item $U_{\alpha}(t,s)$ satisfies $U_{\alpha}(t,s)|_{\B_{\alpha'}} = U_{\alpha'}(t,s)$ whenever $0 \leq \alpha' < \alpha$.
  \item We have for all $\alpha \geq 0$
  \[
   \Vert U_{\alpha}(t,s)\Vert_{L(\K_{\alpha})} \leq \exp\left( \int \limits_{s}^{t}\varkappa(r)\dm r\right),
  \]
  where $\varkappa(r) := e^{\alpha}\int_{X}h(r,x)\sigma(\dm x) + \frac{e^{2\alpha}}{2}\int_{X}\int_{X}\psi(r,x,y)\sigma(\dm x)\sigma(\dm y)$.
  \item For any $T > 0$ and $0 \leq \alpha' < \alpha$, there exists a constant $C(\alpha', \alpha, T) > 0$ such that 
  \begin{align}\label{NONLINEAR:18}
   \Vert U_{\alpha}(t,0)k - U_{\alpha}(s,0)k \Vert_{\K_{\alpha}} \leq C(\alpha', \alpha, T)|t-s|\Vert k \Vert_{\K_{\alpha'}}, \ \ 0 \leq s, t \leq T.
  \end{align}
  Moreover, it is strongly continuously differentiable in $L(\K_{\alpha'}, \K_{\alpha})$ with strong derivatives
  \begin{align*}
   \frac{\partial}{\partial t}U_{\alpha}(t,s)k &= - A_0^{\Delta}(t)U_{\alpha}(t,s)k
   \\ \frac{\partial}{\partial s}U_{\alpha}(t,s)k &= U_{\alpha}(t,s)A_0^{\Delta}(s)k.
  \end{align*}
 \end{enumerate}
 Therefore, conditions A1 -- A4 hold with $\beta = 0$ and any $\lambda > 0$. 
 Now let $B(k,t) := A_1^{\Delta}(t)k + B^{\Delta}(t,k)k$, then it can be easily checked that $B(k,t)$ satisfies B1, for any $r > 0$ and $\lambda > 0$.
 Moreover, for $\alpha \in [\alpha_0, \alpha^*]$ and $t \in [0,T]$ we obtain from Lemma \ref{LEMMA:01}
 \begin{align*}
  \| B(k_0,t) \|_{\mathcal{K}_{\alpha}} &\leq \| A_1^{\Delta}(t)k_0 \|_{\alpha} + \| B^{\Delta}(t,k_0)k_0 \|_{\mathcal{K}_{\alpha}}
  \\ &\leq \frac{C(\alpha_*,T)}{\alpha - \alpha_*} \| k_0 \|_{\mathcal{K}_{\alpha_*}} + C'(\alpha,T) \| k_0\|_{\alpha}^2
  \\ &\leq \frac{C(\alpha_*,T)}{\alpha - \alpha_*} \| k_0 \|_{\mathcal{K}_{\alpha_*}} + \frac{C'(\alpha,T)(\alpha^* - \alpha_*)}{\alpha - \alpha_*} \| k_0 \|_{\alpha_*}^2,
 \end{align*}
 where $C(\alpha_*,T), C'(\alpha,T) \in (0,\infty)$ are continuous in the first argument. Let $\alpha',\alpha$ satisfy $\alpha_* \leq \alpha' < \alpha \leq \alpha^*$, take any $r > 0$ and let $k,k' \in \mathcal{K}_{\alpha'}$
 satisfy $\| k - k_0 \|_{\mathcal{K}_{\alpha'}}, \ \| k' - k_0 \|_{\mathcal{K}_{\alpha'}} \leq r$.
 From Lemma \ref{LEMMA:01} we obtain for all $t \in [0,T]$
 \begin{align*}
  &\ \| B(k,t) - B(k',t) \|_{\mathcal{K}_{\alpha}} 
  \\ &\leq \| A_1^{\Delta}(t) ( k - k') \|_{\mathcal{K}_{\alpha}} 
  + \| B^{\Delta}(t,k)(k-k')\|_{\mathcal{K}_{\alpha}} + \| (B^{\Delta}(t,k) - B^{\Delta}(t,k'))k'\|_{\mathcal{K}_{\alpha}}
 \\ &\leq \frac{C(\alpha',T)}{\alpha - \alpha'} \| k - k' \|_{\mathcal{K}_{\alpha'}} 
  + C'(\alpha,T)\| k - k'\|_{\mathcal{K}_{\alpha}}\| k\|_{\K_{\alpha}} + C''(\alpha',T) \| k - k' \|_{\mathcal{K}_{\alpha'}} \| k \|_{ \mathcal{K}_{\alpha} }
 \\ &\leq \frac{C(\alpha',T) + C'(\alpha,T)(\alpha^* - \alpha_*)(r + \| k_0 \|_{\K_{\alpha_*}}) +  C''(\alpha',T) ( r + \| k_0 \|_{\mathcal{K}_{\alpha_*}}) }{\alpha - \alpha'} \| k - k' \|_{\mathcal{K}_{\alpha'}},
 \end{align*}
 where $C(\alpha,T), C'(\alpha,T)$ are as before, 
 and $C''(\alpha,T) \in (0,\infty)$ is continuous in $\alpha$.
 Hence conditions B2 and B3 are satisfied. 
 
 An application of Theorem \ref{NONLINEARTH:00} and Corollary \ref{NONLINEARTH:01} with $\B_{\alpha} = \E_{\alpha} = \K_{\alpha}$, $\beta = 0$, $\gamma \in (0, 1)$ and $r \in (0,\infty)$
 yields the existence of some constant $\lambda_0(k_0, \alpha_0, \gamma, r, T) > 0$ such that for each $\lambda > \lambda_0$ there exists a unique $(\K_{\alpha})_{\alpha}$-valued solution to \eqref{EQ:120} on the interval $[0, \frac{\alpha^* - \alpha_0}{\lambda})$. The particular form of $\lambda_0$ shows that
 \[
  \lambda_1 := \inf \limits_{\gamma \in (0,1)}\inf \limits_{r \in (0,\infty)} \lambda_0(k_0, \alpha_0, \gamma, r) > 0.
 \]
 Hence we find for each $\lambda > \lambda_1$ a unique classical $(\K_{\alpha})_{\alpha}$-valued solution to \eqref{EQ:120} on the interval $[0, \frac{\alpha - \alpha_0}{\lambda})$. This proves the assertion. 
\end{proof}

\section{Appendix: Proof of Lemma \ref{LEMMA:00}}
Here and below let $\mu \in \mathcal{P}_{\alpha}$ be fixed and let $k_{\mu} \in \K_{\alpha}$ be its correlation function.
For $G \in B_{bs}(\Gamma_0)$ let $KG$ be given by \eqref{EQ:00}. 
Then, by \eqref{EQ:01}, one has
\[
 \| KG \|_{L^1(\Gamma,d\mu)} \leq \| K |G| \|_{L^1(\Gamma,d\mu)} = \int \limits_{\Gamma_0}|G(\eta)|k_{\mu}(\eta)\dm \lambda(\eta) = \| G \|_{L^1(\Gamma_0, k_{\mu}d\lambda)},
\]
i.e. $K: B_{bs}(\Gamma_0) \subset L^1(\Gamma_0, k_{\mu}d\lambda) \longrightarrow L^1(\Gamma,d\mu)$ is a bounded linear operator.
It was shown in \cite{KK02} that there exists a unique extension to a bounded linear opeartor
$K: L^1(\Gamma_0, k_{\mu}d\lambda) \longrightarrow L^1(\Gamma,d\mu)$ such that for any $G \in L^1(\Gamma_0, k_{\mu}d\lambda)$
equation \eqref{EQ:00} holds for $\mu$-a.a. $\gamma$ (where the series is absolutely convergent).

Define, for $G \in B_{bs}(\Gamma_0)$, a linear mapping $\widehat{L}(k_{\mu})G := - \widehat{A}_0(t)G + \widehat{A}_1(t)G + B^{\Delta}(t,k_{\mu})G$ by
\begin{align*}
 \widehat{A}_0(t)G(\eta) &= \Phi(t,\eta)G(\eta) + \sum \limits_{x \in \eta}h(t,x)G(\eta \backslash \{x\}) + \frac{1}{2}\sum \limits_{x \in \eta}\sum \limits_{y \in \eta \backslash \{x\}}\psi(t,x,y)G(\eta \backslash \{x,y\})
 \\ \widehat{A}_1(t)G(\eta) &= - \sum \limits_{x \in \eta}\sum \limits_{y \in \eta \backslash \{x\}}\psi(t,x,y)G(\eta \backslash \{x\}) + \int \limits_{\R^d}a(t,x)G(\eta \cup \{x\})\sigma(\dm x).
\end{align*}
Then one can show that $\widehat{A}_0(t)G, \widehat{A}_1(t)G, B^{\Delta}(t,k_{\mu})G \in L^1(\Gamma_0, k_{\mu}d\lambda)$ from which we conclude that 
$KG, \Phi(t,\cdot), K \widehat{L}(\mu)G \in L^1(\Gamma,d\mu)$.
Similarly to the computations in \cite{FKKO07},
one checks that $K\widehat{L}(\mu)G = L(\mu)KG$.
Moreover, by definition of $\mathcal{FP}(\Gamma)$ and $L^{\Delta}(t,k_{\mu})$ one has 
\begin{align*}
 \langle L(t,\mu)KG, \mu \rangle &= \int \limits_{\Gamma_0}G(\eta)L^{\Delta}(t,k_{\mu})k_{\mu}(\eta)\dm \lambda(\eta), \qquad G \in B_{bs}(\Gamma_0).
\end{align*}
From this and the definition of correlation functions one readily deduces the assertion.

\subsection*{Acknowledgments}
This work was supported by the SFB 701 and IRTG both at Bielefeld University, which is gratefully acknowledged.

\bibliographystyle{amsplain}
\addcontentsline{toc}{section}{\refname}\bibliography{Bibliography}

\end{document}